\documentclass{amsart}
\usepackage{amssymb,amsfonts,amsmath}
\usepackage[all]{xy}
\usepackage[shortlabels]{enumitem}
\usepackage{commath}
\usepackage{esdiff}	
\usepackage[sectionbib]{natbib}
\usepackage[colorlinks,
            linkcolor=blue,      
            anchorcolor=blue, 
            citecolor=blue,]{hyperref}

\newcommand*{\ud}{\mathop{}\!\mathrm{d}}
\usepackage{graphicx} 
 \usepackage{float}
\usepackage{subcaption}
\usepackage{exscale}

\usepackage{relsize}

\renewcommand{\thefootnote}{}

\theoremstyle{theorem}
\newtheorem{theorem}{Theorem}[section]
\newtheorem{lemma}[theorem]{Lemma}
\newtheorem{proposition}[theorem]{Proposition}

\theoremstyle{definition}
\newtheorem{definition}[theorem]{Definition}

\newtheorem{remark}[theorem]{Remark}

\numberwithin{equation}{section}
\numberwithin{figure}{section}

\begin{document}

\title{Constructions of helicoidal minimal surfaces and minimal annuli in $\widetilde{E(2)}$}

\author{Yiming ZANG}

\address{Université de Lorraine, CNRS, IECL, F-54000 Nancy, France}

\email{yiming.zang@univ-lorraine.fr}

\maketitle

\begin{abstract}

In this article, we construct two one-parameter families of properly embedded minimal surfaces in a three-dimensional Lie group $\widetilde{E(2)}$, which is the universal covering of the group of rigid motions of Euclidean plane endowed with a left-invariant Riemannian metric. The first one can be seen as a family of helicoids, while the second one is a family of catenoidal minimal surfaces. The main tool that we use for the construction of these surfaces is a Weierstrass-type representation introduced by Meeks, Mira, Pérez and Ros for minimal surfaces in Lie groups of dimension three. In the end, we study the limit of the catenoidal minimal surfaces. As an application of this limit case, we get a new proof of a half-space theorem for minimal surfaces in $\widetilde{E(2)}$.

\end{abstract}

\renewcommand{\thefootnote}{}
\footnotetext{\text{MSC 2020:} 53A10, 53C42.}
\renewcommand{\thefootnote}{}
\footnotetext{\text{Keywords:} minimal annuli; Weierstrass representation; half-space theorem; $\widetilde{E(2)}$.}

\section{Introduction}

This paper concerns minimal surfaces in Riemannian homogeneous manifolds of dimension three. Among many possible three-dimensional homogeneous manifolds, Lie groups with left-invariant metrics are of great interest. Thanks to a work done by Milnor \cite{milnor1976curvatures}, we obtain a complete classification of three-dimensional metric Lie groups. In particular, if the Lie group is unimodular, then there are only six possible cases: the standard Euclidean space $\mathbb{R}^3$, the special unitary group $\textrm{SU}(2)$, the universal covering of the special linear group $\widetilde{\textrm{SL}}(2,\mathbb{R})$, the solvable group $\textrm{Sol}_3$, the Heisenberg group $\textrm{Nil}_3$ and the universal covering of the group of rigid motions of Euclidean plane $\widetilde{E(2)}$.

In the last twenty years, the theory of minimal surfaces and constant mean curvature surfaces (CMC surfaces) in three-dimensional Riemannian homogeneous manifolds has witnessed significant development. With the evolution of the theory, some explicit examples of minimal surfaces in three-dimensional unimodular Lie groups have also been constructed. Erjavec \cite{erjavec2015minimal} gave some examples of minimal surfaces in $\widetilde{\textrm{SL}}(2,\mathbb{R})$, including a catenoid-type minimal surface. In the article \cite{kokubu1997minimal}, Kokubu also got some rotational minimal surfaces in $\widetilde{\textrm{SL}}(2,\mathbb{R})$. Torralbo \cite{torralbo2012compact} found several compact minimal surfaces in $\textrm{SU}(2)$. Daniel and Hauswirth constructed in \cite{daniel2009half} a family of helicoidal minimal surfaces as well as a family of minimal annuli in $\textrm{Nil}_3$. Desmonts \cite{desmonts2015constructions} discovered a family of minimal annuli in $\textrm{Sol}_3$ as well. However, as far as the author can see, very few examples of minimal surfaces in $\widetilde{E(2)}$ are known until now.

The purpose of this paper is to construct two one-parameter families of properly embedded minimal surfaces in $\widetilde{E(2)}$. The first one is a family of helicoidal minimal surfaces, while the second one is a family of minimal annuli. Our main tool to achieve this goal is the Weierstrass representation for minimal surfaces in three-dimensional metric Lie groups introduced by Meeks, Mira, Pérez and Ros in \cite{meeksconstant} and \cite{meeks2012constant}.

This paper is organized in the following order. In Section 2, we introduce the basic material about the ambient space $\widetilde{E(2)}$, including the model that we use, the left-invariant metric and some curvature properties of this Lie group. In Section 3, we present the main tools such as the Weierstrass representation and the Hopf differential that we will use later in the construction of the two families of minimal surfaces.

In Section 4, we are going to construct the first family of properly embedded minimal surfaces in $\widetilde{E(2)}$ called helicoids. The method that we adopt is inspired by Daniel and Hauswirth in \cite{daniel2009half} and Desmonts in \cite{desmonts2015constructions}. We start from an elliptic PDE that the Gauss map of every minimal surface in $\widetilde{E(2)}$ should satisfy. By a separation of variables, we find a suitable solution to this PDE of which the Hopf differential is in a simple form. Taking advantage of the Weierstrass representation, we obtain a one-parameter family of minimal surfaces in $\widetilde{E(2)}$ with this prescribed Gauss map. In the end, we study some geometric properties of this surface, which share many similarities with the well-known helicoids in $\mathbb{R}^3$.

Section 5 is devoted to the construction of another family of minimal surfaces in $\widetilde{E(2)}$ which can be regarded as an analogue of the classic catenoids in $\mathbb{R}^3$. The principal idea is almost the same with the helicoid case, but there are some difficulties in solving a period problem. In the last section, we study the limit situation of this one-parameter family of minimal annuli. As an application, we give a proof of a half-space theorem for minimal surfaces in $\widetilde{E(2)}$, which can be regarded as a particular case of a more general result of Mazet \cite{mazet2013general}.

\vspace{.5 cm} \noindent{\bf Acknowledgements.}
The author is sincerely grateful to his advisor, Benoît Daniel, for his valuable comments and insightful suggestions during the preparation of this paper.

\section{The Lie group $\widetilde{E(2)}$}

The Euclidean rigid motion group $E(2)$ is defined as the matrix group 
\begin{align*}
   E(2) = \left\{\begin{pmatrix}
      \cos{\theta } & -\sin{\theta} & x \\
      \sin{\theta } & \cos{\theta}  & y \\
      0 & 0 & 1
      \end{pmatrix}\Bigg| ~x,y\in \mathbb{R}, \theta\in\mathbb{S}^1\right\}. 
\end{align*}
This group is not simply connected, so we consider its universal covering group $\widetilde{E(2)}$ which is given by the following definition.
\begin{definition}\label{def 2.1}
The Lie group $\widetilde{E(2)}$ is $\mathbb{R}^3$ with the multiplication 
\begin{align*}
    &(x_1, y_1, z_1)*(x_2, y_2, z_2)\\
    =&(x_1+x_2\cos{z_1}-y_2\sin{z_1}, y_1+x_2\sin{z_1}+y_2\cos{z_1},z_1+z_2)
\end{align*}
for all $(x_1, y_1, z_1),(x_2, y_2, z_2)\in\mathbb{R}^3$. The identity element is $(0,0,0)$ and the inverse element of $(x_1, y_1,z_1)$ is $(-x_1\cos{z_1}-y_1\sin{z_1}, x_1\sin{z_1}-y_1\cos{z_1}, -z_1)$. This Lie group is non-commutative.
\end{definition}

Let us consider a basis $\{E_1, E_2, E_3\}$ of the associated Lie algebra $\mathfrak{e}(2)$ (consisting of left-invariant vector fields) given by 
\begin{align*}
    E_1 = \frac{1}{\lambda_1}(\cos{z}\partial_x + \sin{z}\partial_y), \
    E_2 = \frac{1}{\lambda_2}(-\sin{z}\partial_x + \cos{z}\partial_y), \
    E_3 = \frac{1}{\lambda_3}\partial_z,
\end{align*}
with $\lambda_1, \lambda_2, \lambda_3$ positive real numbers. It is easy to check that 
\begin{align}\label{2.1}
    [E_1,E_2] = 0, \ 
    [E_2,E_3] = \frac{\lambda_1}{\lambda_2\lambda_3}E_1, \ 
    [E_3,E_1] = \frac{\lambda_2}{\lambda_1\lambda_3}E_2.
\end{align}
Then we are able to find a left-invariant Riemannian metric on $\widetilde{E(2)}$ such that $\{E_1, E_2, E_3\}$ becomes a left-invariant orthonormal frame. This metric could be determined to be 
\begin{align}\label{2.2}
    g(\lambda_1, \lambda_2, \lambda_3) = \lambda_1^{2}(\cos{z}{\ud}x + \sin{z}{\ud}y)^2 +\lambda_2^{2}(-\sin{z}{\ud}x + \cos{z}{\ud}y)^2 + \lambda_3^{2}{\ud}z^2.
\end{align}
The Lie group $\widetilde{E(2)}$ together with this left-invariant metric becomes a homogeneous manifold.

From now on, parentheses are used to denote the coordinates of a vector field on $\widetilde{E(2)}$ in the frame $\{\partial_x, \partial_y, \partial_z\}$, while brackets are reserved to express the coordinates in the frame $\{E_1, E_2, E_3\}$.

Patrangenaru proved the following theorem which gives a classification of left-invariant metrics on $\widetilde{E(2)}$ (see \cite{inoguchi2007parallel}, \cite{patrangenaru1996classifying}).
\begin{theorem}\label{thm 2.1}
Any left-invariant metrics on $\widetilde{E(2)}$ is isometric to one of the metric $g(\lambda_1, \lambda_2, \lambda_3)$ with $\lambda_1>\lambda_2>0$ and $\lambda_3 = \frac{1}{\lambda_1\lambda_2}$, or $\lambda_1=\lambda_2=\lambda_3=1$.
\end{theorem}
In the case when $\lambda_1=\lambda_2=\lambda_3=1$, it also satisfies $\lambda_3 = \frac{1}{\lambda_1\lambda_2}$. Hence metric (\ref{2.2}) can be written as a two-parameter family
\begin{align}\label{2.3}
    g(\lambda_1, \lambda_2) = \lambda_1^{2}(\cos{z}{\ud}x + \sin{z}{\ud}y)^2 +\lambda_2^{2}(-\sin{z}{\ud}x + \cos{z}{\ud}y)^2 + \frac{1}{\lambda_1^{2}\lambda_2^{2}}{\ud}z^2.
\end{align}
In a similar manner, equations (\ref{2.1}) take the form 
\begin{align}\label{2.4}
    [E_1,E_2] = 0, \ 
    [E_2,E_3] = \lambda_1^{2}E_1, \ 
    [E_3,E_1] = \lambda_2^{2}E_2.
\end{align}

\begin{remark}\label{rmk 2.1}
Milnor has shown in his article \cite{milnor1976curvatures} that $\widetilde{E(2)}$ equipped with this left-invariant metric is a 3-dimensional unimodular Lie group. If we choose the orientation such that $\{E_1, E_2, E_3\}$ is positively oriented, then there is a uniquely well-defined self-adjoint linear mapping $L:\mathfrak{e}(2)\rightarrow \mathfrak{e}(2)$ with respect to the metric which satisfies $L(u\times v)=[u,v]$ for all $u,v\in\mathfrak{e}(2)$, where $u \times v$ is the cross product. Relations in (\ref{2.4}) imply that $\lambda_1^{2}, \lambda_2^{2}, 0$ are eigenvalues of $L$. Moreover, this basis $\{E_1, E_2, E_3\}$ diagonalizes the Ricci quadratic form $Ric$. If we denote 
\begin{align}\label{2.5}
    \mu_1 = \frac{\lambda_2^{2}-\lambda_1^{2}}{2}, \ 
    \mu_2 = \frac{\lambda_1^{2}-\lambda_2^{2}}{2}, \ 
    \mu_3 = \frac{\lambda_1^{2}+\lambda_2^{2}}{2},
\end{align}
then the three principal Ricci curvatures are 
\begin{align*}
    Ric(E_1) = 2\mu_2\mu_3, \ Ric(E_2) = 2\mu_1\mu_3, \ Ric(E_3) = 3\mu_1\mu_2. 
\end{align*}
It follows that the scalar curvature $S$ is given by 
\begin{align*}
    S = \sum_{i=1}^{3}Ric(E_i) = 2\mu_2\mu_3.
\end{align*}
As a consequence, $g(1,1)$ is the only left-invariant flat metric on $\widetilde{E(2)}$. Any other metric $g(\lambda_1, \lambda_2)$ with $\lambda_1>\lambda_2>0$ has Ricci curvature form of signature $(+,-,-)$ and strictly negative scalar curvature $S$. In particular, $\widetilde{E(2)}$ together with the unique flat left-invariant metric is isometric to the Euclidean space $\mathbb{R}^3$, hence its isometry group has dimension $6$. However, the isometry group of $\widetilde{E(2)}$ equipped with a metric $g(\lambda_1, \lambda_2)$ with $\lambda_1>\lambda_2>0$ is of dimension $3$ (see \cite{ha2012isometry}).
\end{remark}

With these notations, we are able to compute the Levi-Civita connection $\nabla$ of $\widetilde{E(2)}$ associated to the metric $g(\lambda_1, \lambda_2)$ in the frame  $\{E_1, E_2, E_3\}$. After a simple computation, we obtain  
\begin{align*}
   &\nabla_{E_1}E_1 = 0, \ 
   \nabla_{E_1}E_2 = \mu_1E_3, \  
   \nabla_{E_1}E_3 = \mu_2E_2,\\
   &\nabla_{E_2}E_1 = \mu_1E_3, \ 
   \nabla_{E_2}E_2 = 0, \ 
   \nabla_{E_2}E_3 = \mu_2E_1, \\ 
   &\nabla_{E_3}E_1 = \mu_3E_2, \ 
   \nabla_{E_3}E_2 = -\mu_3E_1, \ 
   \nabla_{E_3}E_3 = 0.
\end{align*}

 \section{The Gauss map and the Weierstrass representation}

Let $\Sigma$ be a Riemann surface and $z=u+iv$ a local complex coordinate on $\Sigma$. Let us consider a conformal immersion $X:\Sigma\rightarrow\widetilde{E(2)}$. We denote by $N:\Sigma\rightarrow T\widetilde{E(2)}$ the unit normal vector field to $\Sigma$, where $T\widetilde{E(2)}$ stands for the tangent bundle of $\widetilde{E(2)}$. For a given point $p\in\Sigma$, we may define a mapping $G:\Sigma\rightarrow\mathbb{S}^2\subset\mathbb{R}^3\equiv\mathfrak{e}(2)$ by 
\begin{align*}
    G(p) = \sum_{i=1}^{3}N_i E_i,
\end{align*}
with $N_i=\langle N_p, (E_i)_p\rangle$. This mapping $G:\Sigma\rightarrow\mathbb{S}^2$ is called the left-invariant Gauss map of $\Sigma$. 
\begin{definition}\label{def 3.1}
    The Gauss map of the immersion $X:\Sigma\rightarrow\widetilde{E(2)}$ is the mapping
    \begin{align*}
        g = \pi\circ G:\Sigma\rightarrow\overline{\mathbb{C}}=\mathbb{C}\cup \{\infty\},
    \end{align*}
    where $\pi$ is the stereographic projection from the south pole, i.e.,
    \begin{align*}
        g = \frac{N_1+iN_2}{1+N_3},
    \end{align*}
    and
    \begin{align*}
        G = \frac{1}{1+|g|^2}\begin{bmatrix}
        g+\Bar{g}\\
        i(\Bar{g}-g)\\
        1-|g|^2
        \end{bmatrix}.
    \end{align*}
\end{definition}

In order to introduce the Weierstrass representation, we need the notion of $H$-potential (see \cite{meeks2012constant}).  
\begin{definition}\label{def 3.2}
    Assuming that $\mu_1,\mu_2,\mu_3\in\mathbb{R}$ are as in (\ref{2.5}) and $H\geq 0$, the $H$-potential for $\widetilde{E(2)}$ is defined to be the mapping $R:\overline{\mathbb{C}}\rightarrow\overline{\mathbb{C}}$ given by
    \begin{align}\label{3.1}
        R(q) = H(1+|q|^2)^2 - \frac{i}{2}\left[\mu_2|1 + q^2|^2 + \mu_1|1-q^2|^2 + 4\mu_3|q|^2\right].
    \end{align}
\end{definition}
Usually, we write $R_q=\partial_{q}R$ and $R_{\Bar{q}}=\partial_{\Bar{q}}R$ for $q\in\mathbb{C}$. With all these notations, we can give the following Weierstrass representation proved by Meeks, Mira, Pérez and Ros (see \cite{meeksconstant} or Theorem 3.15 of \cite{meeks2012constant}).
\begin{theorem}\label{thm 3.1}
Let $X:\Sigma\rightarrow\widetilde{E(2)}$ be a conformally immersed surface of constant mean curvature $H\geq 0$ in $\widetilde{E(2)}$ with left-invariant Gauss map $G:\Sigma\rightarrow\mathbb{S}^2$. Assume that the $H$-potential $R$ does not vanish on $G(\Sigma)$, then\\
(1). The Gauss map $g$ of the immersion $X$ satisfies the elliptic PDE
\begin{align}\label{3.2}
    g_{z\Bar{z}} = \frac{R_q}{R}(g)g_{z}g_{\Bar{z}} + \left(\frac{R_{\Bar{q}}}{R} - \frac{\overline{R_{q}}}{\overline{R}}\right)(g)|g_{z}|^2.
\end{align}
(2). The expression of $X_{z} = \sum_{i=1}^{3}A_i E_i$ can be written as
\begin{align}\label{3.3}
    A_1 = \frac{\eta}{4}\left(\Bar{g} - \frac{1}{\Bar{g}}\right),\quad
    A_2 = \frac{i\eta}{4}\left(\Bar{g}+ \frac{1}{\Bar{g}}\right),\quad
    A_3 = \frac{\eta}{2},\quad 
    \eta = \frac{4\Bar{g}g_{z}}{R(g)}.
\end{align}
Moreover, the induced metric ${\ud}s^2$ on $\Sigma$ is given by 
\begin{align}\label{3.4}
    {\ud}s^2 = \frac{4(1+|g|^2)^2|g_{z}|^2}{|R(g)|^2}|\rm{d}z|^2.
\end{align}
\end{theorem}

The reverse of this theorem is also true. In fact, we have the next theorem (see Corollary 3.16 of \cite{meeks2012constant}).

\begin{theorem}\label{thm 3.2}
Let $\Sigma$ be a simply connected Riemann surface and $g:\Sigma\rightarrow\overline{\mathbb{C}}$ a solution of the PDE (\ref{3.2}). If the $H$-potential $R$ has no zeros in $g(\Sigma)\subset\overline{\mathbb{C}}$ and the mapping $g$ is nowhere antiholomorphic, then up to a left translation, there exists a unique conformal immersion $X:\Sigma\rightarrow\widetilde{E(2)}$ with constant mean curvature $H$ and the prescribed Gauss map $g$.
\end{theorem}

In this article, we are mainly interested in studying minimal immersions ($H=0$) in $\widetilde{E(2)}$, hence the $0$-potential $R$ in (\ref{3.1}) can be explicitly determined to be 
\begin{align}\label{3.5}
    R(q) = - \frac{i}{2}\left[\lambda_1^2(q + \Bar{q})^2 - \lambda_2^2(q - \Bar{q})^2\right].
\end{align}
Then we may plug (\ref{3.5}) into (\ref{3.2}) to get the elliptic PDE for this minimal immersion:
\begin{align}\label{3.6}
    g_{z\Bar{z}} = \frac{2\left[ \lambda_1^2(g + \Bar{g}) - \lambda_2^2(g - \Bar{g})\right]}{\lambda_1^2(g + \Bar{g})^2 - \lambda_2^2(g - \Bar{g})^2}g_{z}g_{\Bar{z}}.
\end{align}
\begin{remark}\label{rmk 3.1}
Equation (\ref{3.6}) is actually the harmonic map equation for the mapping $g:\Sigma\rightarrow(\overline{\mathbb{C}}, {\ud}\sigma^2)$, with the metric 
\begin{align*}
    {\ud}\sigma^2 = \frac{|\rm{d}w|^2}{|\lambda_1^2(w + \Bar{w})^2 - \lambda_2^2(w - \Bar{w})^2|}.
\end{align*}
This metric ${\ud}\sigma^2$ is a singular metric which is not well-defined at $0$ and $\infty$. We refer to \cite{jost1984harmonic} and \cite{schoen1978univalent} for more details.
\end{remark}

\begin{definition}\label{def 3.3}
    The Hopf differential associated to $g$ is defined to be the quadratic differential
    \begin{align}\label{3.7}
        Q = \frac{g_{z}\Bar{g}_{z}}{\lambda_1^2(g + \Bar{g})^2 - \lambda_2^2(g - \Bar{g})^2}{\ud}z^2.
    \end{align}
\end{definition}
\begin{remark}\label{rmk 3.2}
It could be easily checked that $Q$ does not depend on the choice of the complex coordinates. Moreover, since $g:\Sigma\rightarrow(\overline{\mathbb{C}}, {\ud}\sigma^2)$ is a harmonic mapping, we know that the Hopf differential $Q$ is holomorphic (see \cite{jost1984harmonic},\cite{schoen1978univalent}). 
\end{remark}

\section{Helicoidal minimal surfaces in $\widetilde{E(2)}$}

In this section we construct a one-parameter family of helicoidal minimal surfaces in $\widetilde{E(2)}$. First of all, we define the helicoidal minimal surface to be a minimal surface containing the $z$-axis which is invariant under the left multiplication by an element $(0,0,T)\in\widetilde{E(2)}$ with some fixed $T\neq 0$ and whose intersection with each plane corresponding to $\{z = constant\}$ is a straight line.

Let us define a mapping $g:\mathbb{C}\rightarrow\overline{\mathbb{C}}$ by 
\begin{align*}
    g(z = u+iv) = e^{-\lambda_1u}e^{ib(v)},
\end{align*}
where the function $b$ satisfies the ODE
\begin{align}\label{4.1}
    b'(v) = \sqrt{\lambda_1^2 - K\left(\lambda_1^2\cos^2{b(v)} + \lambda_2^2\sin^2{b(v)}\right)},
    \quad b(0) = 0, 
\end{align}
with $K\in (-1,1)$.

\begin{proposition}\label{prop 4.1}
The function $b$ is well-defined and satisfies the following properties:\\
(1). The function $b$ is an increasing diffeomorphism from $\mathbb{R}$ to $\mathbb{R}$.\\
(2). $b$ is an odd function.\\
(3). There exists a real number $W>0$ such that 
\begin{align*}
    b(v + W) = b(v) + \pi, \quad \forall v\in\mathbb{R}.
\end{align*}
(4). $b(kW) = k\pi$ for all $k\in\mathbb{Z}$.\\
(5). $b(k\frac{W}{2}) = k\frac{\pi}{2}$ for all $k\in2\mathbb{Z}+1$.
\end{proposition}
\begin{proof}
(1). Since $K\in(-1,1)$ and $\lambda_2^2 \leq \lambda_1^2\cos^2{b(v)} + \lambda_2^2\sin^2{b(v)}\leq \lambda_1^2$, there exists a positive constant $r$ such that $r\leq \lambda_1^2 - K\left(\lambda_1^2\cos^2{b(v)} + \lambda_2^2\sin^2{b(v)}\right)<2\lambda_1^2$, hence $\sqrt{r}\leq b'(v)<\sqrt{2}\lambda_1$. By the Cauchy-Lipschitz theorem, we know the function $b$ is well-defined on $\mathbb{R}$. Moreover, $\lim\limits_{v\to+\infty} b(v) = +\infty$ and $\lim\limits_{v\to-\infty} b(v) = -\infty$ show that $b$ is an increasing diffeomorphism from $\mathbb{R}$ to $\mathbb{R}$.\\
(2). The function $\hat{b}:=-b(-v)$ satisfies ODE (\ref{4.1}) and the initial condition $\hat{b}(0)=0$, thus $\hat{b}=b$ and $b$ is an odd function.\\
(3). It follows from (1) that there exists a positive real number $W$ such that $b(W)=\pi$. Let us consider the function $\Tilde{b}(v):=b(v+W)-\pi$. We may easily verify that $\Tilde{b}$ also satisfies ODE (\ref{4.1}) and the initial condition $\Tilde{b}(0)=0$, hence $\Tilde{b}=b$.\\
(4). This follows directly from (3).\\
(5). $b(\frac{W}{2}) = b(-\frac{W}{2} + W) = -b(\frac{W}{2}) + \pi$, thus $b(\frac{W}{2}) = \frac{\pi}{2}$. Then the result can be deduced from (3). 
\end{proof}

\begin{proposition}\label{prop 4.2}
The mapping $g$ satisfies 
\begin{align*}
    g_{z\Bar{z}} = \frac{2\left[ \lambda_1^2(g + \Bar{g}) - \lambda_2^2(g - \Bar{g})\right]}{\lambda_1^2(g + \Bar{g})^2 - \lambda_2^2(g - \Bar{g})^2}g_{z}g_{\Bar{z}},
\end{align*}
and its Hopf differential is $Q = \frac{K}{16}{\ud}z^2$.
\end{proposition}
\begin{proof}
A direct computation shows 
\begin{align*}
    g_{z} = \frac{1}{2}\left(b'(v)-\lambda_1\right)e^{-\lambda_1u}e^{ib(v)},
\end{align*}
\begin{align*}
   g_{\Bar{z}} = -\frac{1}{2}\left(b'(v)+\lambda_1\right)e^{-\lambda_1u}e^{ib(v)}, 
\end{align*}
\begin{align*}
    \Bar{g}_{z} = -\frac{1}{2}\left(b'(v)+\lambda_1\right)e^{-\lambda_1u}e^{-ib(v)},
\end{align*}
\begin{align*}
   g_{z\Bar{z}} = \frac{1}{4}\left[\left(-b'^2(v)+\lambda_1^2\right) + ib''(v)\right]e^{-\lambda_1u}e^{ib(v)}. 
\end{align*}    
Moreover, by differentiating (\ref{4.1}) we obtain
\begin{align}\label{4.2}
    b''(v) = K(\lambda_1^2 - \lambda_2^2)\sin{b(v)}\cos{b(v)}.
\end{align}
Taking advantage of these formulas, we may see that the equation is satisfied. Then the expression of $Q$ follows immediately from (\ref{3.7}). 
\end{proof}

Comparing Proposition \ref{prop 4.2} and Theorem \ref{thm 3.1}, we are inspired to utilise the Weierstrass representation to find a conformal minimal immersion $X(u+iv) = (x_1, x_2, x_3)$ into $\widetilde{E(2)}$ whose Gauss map is $g$. In order to do this, we need the $0$-potential $R$ which takes the form
\begin{align*}
    R(g) = -2ie^{-2\lambda_1u}\left(\lambda_1^2\cos^2{b(v)} + \lambda_2^2\sin^2{b(v)}\right).
\end{align*}
Therefore, we get $\eta = \frac{4\Bar{g}g_{z}}{R(g)} = -i\frac{K}{\lambda_1 + b'(v)}$ and $X_{z} = \sum_{i=1}^{3}{x_i}_{z} \partial_{x_i} = \sum_{i=1}^{3}A_i E_i$ with
\begin{align}\label{4.3}
    A_1 = -\frac{K}{2(\lambda_1 + b'(v))}\left(\cosh{(-\lambda_1u)}\sin{b(v)} + i\sinh{(-\lambda_1u)}\cos{b(v)}\right),
\end{align}
\begin{align}\label{4.4}
    A_2 = \frac{K}{2(\lambda_1 + b'(v))}\left(\cosh{(-\lambda_1u)}\cos{b(v)} - i\sinh{(-\lambda_1u)}\sin{b(v)}\right),
\end{align}
\begin{align}\label{4.5}
    A_3 = \frac{-iK}{2(\lambda_1 + b'(v))}.
\end{align}
The two expressions of $X_z$ show that 
%\begin{align}\label{4.6}
%    \left\{
%    \begin{array}{ll}
%      & {x_1}_z = \frac{1}{\lambda_1}A_1\cos{x_3} - \frac{1}{\lambda_2}A_2\sin{x_3},\\
%      & {x_2}_z = \frac{1}{\lambda_1}A_1\sin{x_3} + \frac{1}{\lambda_2}A_2\cos{x_3},\\
%      & {x_3}_z = \lambda_1\lambda_2A_3.
%    \end{array}
%  \right.
%\end{align}

 \begin{equation}\label{4.6}
\left\{
\begin{aligned}
  & ~{x_1}_z = \frac{1}{\lambda_1}A_1\cos{x_3} - \frac{1}{\lambda_2}A_2\sin{x_3},\\
  & ~{x_2}_z = \frac{1}{\lambda_1}A_1\sin{x_3} + \frac{1}{\lambda_2}A_2\cos{x_3},\\
  & ~{x_3}_z = \lambda_1\lambda_2A_3.
\end{aligned}
\right.
\end{equation}
We can see immediately that $x_3$ is actually a one-variable function of $v$ which satisfies
\begin{align}\label{4.7}
    x_3'(v) = \frac{\lambda_1\lambda_2K}{\lambda_1+b'(v)}.
\end{align}

\begin{remark}\label{rmk 4.1}
If $K=0$, then $x_3$ is a constant, and the image of $X$ reduces to a point. Thus we will exclude this case in the sequel.
\end{remark}

\begin{proposition}\label{prop 4.3}
If we set $x_3(0)=0$, then \\
(1). The function $x_3$ is a well-defined bijection from $\mathbb{R}$ to $\mathbb{R}$.\\
(2). The function $x_3$ is odd.\\
(3). For all $v\in\mathbb{R}$, we have
\begin{align*}
    x_3(v+W) = x_3(v) + x_3(W).
\end{align*}
\end{proposition}
\begin{proof}
(1). As we have seen in the proof of Proposition \ref{prop 4.1}, $\sqrt{r}\leq b'(v)<\sqrt{2}\lambda_1$ for some positive number $r$, hence $x_3'(v)$ is bounded by two constants with the same sign of $K$. Consequently, $x_3$ is a well-defined bijection from $\mathbb{R}$ to $\mathbb{R}$.\\
(2). Since the function $b$ is an odd function, $b'$ is even. Thus $x_3'$ is also even. In combination with the condition $x_3(0)=0$, we know that $x_3$ is odd.\\
(3). This is because $x_3'(v) = x_3'(v+W)$ and $x_3(W) - x_3(0) = x_3(W)$. 
\end{proof}
By solving the equations in (\ref{4.6}), we get 
\begin{align*}
    x_1(u+iv) = \frac{1}{\lambda_1^2\lambda_2}\left(\frac{1}{\lambda_1}\cos{x_3(v)}\sin{b(v)} + \frac{1}{\lambda_2}\sin{x_3(v)}\cos{b(v)}\right)x_3'(v)\sinh{(-\lambda_1u)},
\end{align*}
\begin{align*}
    x_2(u+iv) = \frac{1}{\lambda_1^2\lambda_2}\left(\frac{1}{\lambda_1}\sin{x_3(v)}\sin{b(v)} - \frac{1}{\lambda_2}\cos{x_3(v)}\cos{b(v)}\right)x_3'(v)\sinh{(-\lambda_1u)}.
\end{align*}

We may now summarise these results as the following theorem.
\begin{theorem}\label{thm 4.1}
Let $K$ be a non-zero real number with $|K|<1$, and $b$ be the function defined by ODE (\ref{4.1}). The function $x_3$ is defined by ODE (\ref{4.7}) with the condition $x_3(0)=0$. Then the mapping 
\begin{align*}
    X(u+iv)=\begin{pmatrix}
        \frac{1}{\lambda_1^2\lambda_2}\left(\frac{1}{\lambda_1}\cos{x_3(v)}\sin{b(v)} + \frac{1}{\lambda_2}\sin{x_3(v)}\cos{b(v)}\right)x_3'(v)\sinh{(-\lambda_1u)}\\
        \frac{1}{\lambda_1^2\lambda_2}\left(\frac{1}{\lambda_1}\sin{x_3(v)}\sin{b(v)} - \frac{1}{\lambda_2}\cos{x_3(v)}\cos{b(v)}\right)x_3'(v)\sinh{(-\lambda_1u)}\\
        x_3(v)
        \end{pmatrix}
\end{align*}
gives a conformal minimal immersion from $\mathbb{C}$ to $\widetilde{E(2)}$ whose Gauss map is 
\begin{align*}
    g(z = u+iv) = e^{-\lambda_1u}e^{ib(v)}.
\end{align*}
\end{theorem}

Now we are going to prove that the surfaces constructed above are exactly the desired helicoidal surfaces.
\begin{theorem}\label{thm 4.2}
Let $X:\mathbb{C}\rightarrow\widetilde{E(2)}$ be the minimal immersion given by Theorem \ref{thm 4.1}. Then\\
(1). The $x_3$-axis is contained in the image of $X$.\\
(2). The intersection of the surface with each plane $\{x_3 = constant\}$ is a straight line.\\
(3). The surface is invariant under the left multiplication by $(0,0,2x_3(W))$, i.e., 
\begin{align}\label{4.8}
    (0,0,2x_3(W))*X(u+iv) = X(u+i(v+2W)),
\end{align}
for all $u,v\in\mathbb{R}$.
\end{theorem}
\begin{proof}
(1). Let us take $u=0$, then we see $x_1(iv) = x_2(iv) = 0$. By Proposition \ref{prop 4.3}, it is known that $x_3$ is a bijection from $\mathbb{R}$ to $\mathbb{R}$, hence the $x_3$-axis is contained in the image of $X$.\\
(2). Let us consider the plane $\{x_3 = C\}$, where $C\in\mathbb{R}$ is a constant. Then there exists a unique $v_0\in\mathbb{R}$ such that $x_3(v_0)=C$. The intersection of the surface with this plane is the curve $L = \{(k_1\sinh{(-\lambda_1u)},k_2\sinh{(-\lambda_1u)},C)\mid u\in\mathbb{R}\}$, where 
\begin{align*}
    k_1 = \frac{1}{\lambda_1^2\lambda_2}\left(\frac{1}{\lambda_1}\cos{x_3(v_0)}\sin{b(v_0)} + \frac{1}{\lambda_2}\sin{x_3(v_0)}\cos{b(v_0)}\right)x_3'(v_0),
\end{align*}
\begin{align*}
    k_1 = \frac{1}{\lambda_1^2\lambda_2}\left(\frac{1}{\lambda_1}\sin{x_3(v_0)}\sin{b(v_0)} - \frac{1}{\lambda_2}\cos{x_3(v_0)}\cos{b(v_0)}\right)x_3'(v_0).
\end{align*}
Since $\sinh(-\lambda_1u)$ is bijective from $\mathbb{R}$ to $\mathbb{R}$ and 
\begin{align*}
    k_1^2 + k_2^2 = \frac{1}{\lambda_1^4\lambda_2^2}x_3'^2(v_0)\left(\frac{1}{\lambda_1^2}\sin^2{b(v_0)} + \frac{1}{\lambda_2^2}\cos^2{b(v_0)}\right)>0,
\end{align*}
$L$ is therefore a straight line.\\
(3). This is a direct verification by using Proposition \ref{prop 4.3}.
\end{proof}

\begin{definition}\label{def 4.1}
   The helicoidal minimal surface given by Theorem \ref{thm 4.1} is called a helicoid of parameter $K$ and we denote it by $\mathfrak{H}_K$. The set $\{X(u+iv)\mid u\in\mathbb{R}, v\in[0, 2W)\}$ is called the fundamental piece of the helicoid $\mathfrak{H}_K$. 
\end{definition}

\begin{remark}\label{rmk 4.2}
The surface $\mathfrak{H}_K$ is embedded in $\widetilde{E(2)}$ because $x_3$ is a bijection from $\mathbb{R}$ to $\mathbb{R}$. Moreover, it is even properly embedded.
\end{remark}

\begin{remark}\label{rmk 4.3}
It is easy to see that 
%\begin{align*}
%    \left\{
%    \begin{array}{ll}
%      & x_1(-u+iv) = -x_1(u+iv),\\
%      & x_2(-u+iv) = -x_2(u+iv),\\
%      & x_3(-u+iv) = x_3(u+iv),
%    \end{array}
%  \right.
%\end{align*}
 \begin{equation*}
\left\{
\begin{aligned}
  & ~x_1(-u+iv) = -x_1(u+iv),\\
  & ~x_2(-u+iv) = -x_2(u+iv),\\
  & ~x_3(-u+iv) = x_3(u+iv),
\end{aligned}
\right.
\end{equation*}
hence the helicoid $\mathfrak{H}_K$ is symmetric by rotation of angle $\pi$ around the $x_3$-axis. Similarly, we get 
%\begin{align*}
%    \left\{
%    \begin{array}{ll}
%      & x_1(u-iv) = -x_1(u+iv),\\
%      & x_2(u-iv) = x_2(u+iv),\\
%      & x_3(u-iv) = -x_3(u+iv).
%    \end{array}
%  \right.
%\end{align*}
 \begin{equation*}
\left\{
\begin{aligned}
  & ~x_1(u-iv) = -x_1(u+iv),\\
  & ~x_2(u-iv) = x_2(u+iv),\\
  & ~x_3(u-iv) = -x_3(u+iv).
\end{aligned}
\right.
\end{equation*}
It means $\mathfrak{H}_K$ is symmetric by rotation of angle $\pi$ around the $x_2$-axis. As a consequence, $\mathfrak{H}_K$ is also symmetric by rotation of angle $\pi$ around the $x_1$-axis.
\end{remark}

\begin{proposition}\label{prop 4.4}
For every real number $T\neq 0$, there exists a $\mathfrak{H}_K$ whose period is $T$.
\end{proposition}
\begin{proof}
The period of $\mathfrak{H}_K$ is 
\begin{align*}
    & 2{x_3}_K(W) = 2\mathlarger{\int}_{0}^{W} \frac{\lambda_1\lambda_2K}{\lambda_1 + b'(v)}\,{\ud}v\\
    = & \mathlarger{\int}_{0}^{\pi} \frac{2\lambda_1\lambda_2K{\ud}w}{\sqrt{\lambda_1^2 - K(\lambda_1^2\cos^2{w} + \lambda_2^2\sin^2{w})}\left(\lambda_1 + \sqrt{\lambda_1^2 - K(\lambda_1^2\cos^2{w} + \lambda_2^2\sin^2{w})}\right)}.
\end{align*}
Let us regard ${x_3}_K$ as a function of $K$, then we get
\begin{align*}
    \frac{\partial{x_3}_K(W)}{\partial K} = \frac{\lambda_1^2\lambda_2}{2}\mathlarger{\int}_{0}^{\pi} \frac{{\ud}w}{\sqrt{\lambda_1^2 - K(\lambda_1^2\cos^2{w} + \lambda_2^2\sin^2{w})}}>0,
\end{align*}
hence the function $K\mapsto{x_3}_K(W)$ is strictly increasing. In addition, we have ${x_3}_0(W)=0$ and
\begin{align*}
    {x_3}_1 & = \lambda_1\lambda_2\mathlarger{\int}_{0}^{\pi}\frac{{\ud}w}{\sqrt{\lambda_1^2-\lambda_2^2}\sin{w}(\sqrt{\lambda_1^2-\lambda_2^2}\sin{w}+\lambda_1)}\\
    & = \frac{\lambda_2}{\sqrt{\lambda_1^2-\lambda_2^2}}\mathlarger{\int}_{0}^{\infty}\frac{1+t^2}{t\left(t^2 + 2\frac{\sqrt{\lambda_1^2-\lambda_2^2}}{\lambda_1}t + 1\right)}\,{\ud}t\\
    &\geq \frac{\lambda_2}{\sqrt{\lambda_1^2-\lambda_2^2}}\mathlarger{\int}_{0}^{\infty}\frac{t^2 + 1}{t(t+1)^2}\,{\ud}t = +\infty,
\end{align*}
thus ${x_3}_K(W)$ is a bijection from $(0,1)$ to $(0,+\infty)$. Since the inverse element of $(0,0,T)$ is $(0,0,-T)$, every non-zero real number $T$ could be the period for some $\mathfrak{H}_K$.
\end{proof}

Now, we use Theorem \ref{thm 3.1} again to compute the metric induced by this minimal immersion $X$, which is
\begin{align}\label{4.9}
    {\ud}s^2 = \frac{K^2\cosh^2{(-\lambda_1u)}}{(\lambda_1+b'(v))^2}|\mathrm{d}z|^2.
\end{align}
With the help of this formula, we are able to investigate some curvature properties of $\mathfrak{H}_K$. As the first application, we obtain an explicit expression for the Gauss curevature $\mathcal{K}$:
\begin{align*}
    \mathcal{K} = & - \frac{\lambda_1^2(\lambda_1+b'(v))^2}{K^2\cosh^4{(-\lambda_1u)}} - \frac{(\lambda_1^2-\lambda_2^2)^2\sin^2{2b(v)}}{4\cosh^2{(-\lambda_1u)}} - \frac{\lambda_1(\lambda_1^2-\lambda_2^2)\cos{2b(v)}(b'(v)+\lambda_1)}{K\cosh^2{(-\lambda_1u)}}\\
    & - \frac{(\lambda_1^2-\lambda_2^2)\left(\lambda_1^2\cos^2{b(v)} + \lambda_2^2\sin^2{b(v)}\right)\cos{2b(v)}}{\cosh^2{(-\lambda_1u)}}.
\end{align*}
Secondly, it should be interesting to study the total absolute curvature of the fundamental piece $\Omega$ of $\mathfrak{H}_K$. If we denote by ${\ud}A$ the area element of $\mathfrak{H}_K$, then 
\begin{align*}
    |\mathcal{K}|{\ud}A = \left|\frac{\lambda_1^2}{\cosh^2{(-\lambda_1u)}} - \frac{K(\lambda_1^2-\lambda_2^2)b'(v)\cos{2b(v)}}{\lambda_1+b'(v)} + \frac{K^2(\lambda_1^2-\lambda_2^2)^2\sin^2{2b(v)}}{4(\lambda_1+b'(v))^2}\right|{\ud}u \mathrm{d}v.
\end{align*}
(1). In the case when $\lambda_1=\lambda_2=1$, we can see that 
\begin{align*}
    \int_{\Omega}|\mathcal{K}|{\ud}A = 4W < +\infty.
\end{align*}
In fact, the constant $W$ could be determined to be $\frac{\pi}{\sqrt{1-K}}$, thus the finite total absolute curvature is $\frac{4\pi}{\sqrt{1-K}}$. It is worth mentioning that this case is isometric to the classic helicoids in $\mathbb{R}^3$.\\
(2). When $\lambda_1 > \lambda_2 > 0$, we have 
\begin{align*}
    \int_{\Omega}|\mathcal{K}|{\ud}A = +\infty.
\end{align*}

\section{Minimal annuli in $\widetilde{E(2)}$}

In this section, we plan to construct minimal annuli in $\widetilde{E(2)}$. As in the previous section, we begin with defining a mapping $g:\mathbb{C}\rightarrow\overline{\mathbb{C}}$ by 
\begin{align*}
    g(z = u+iv) = e^{f(u)+cv}e^{i\varphi(u)},
\end{align*}
where $c$ is a positive real number. Here the function $\varphi$ satisfies the ODE
\begin{align*}
     \varphi'^2(u) = & c^2 + 2\cos{\theta}\left(\lambda_1^2\cos^2{\varphi(u)} + \lambda_2^2\sin^2{\varphi(u)}\right)\\
     & - \frac{\sin^2{\theta}}{c^2}\left(\lambda_1^2\cos^2{\varphi(u)} + \lambda_2^2\sin^2{\varphi(u)}\right)^2,
\end{align*}
with $\theta\in\mathbb{R}$ and the initial condition $\varphi(0)= 0$. The other function $f$ is defined by the ODE
\begin{align*}
    f'(u) = \frac{\sin{\theta}}{c}\left(\lambda_1^2\cos^2{\varphi(u)} + \lambda_2^2\sin^2{\varphi(u)}\right), \quad f(0)=0.
\end{align*}
In the sequel, we will set 
\begin{align*}
    A = f(u) +cv, \quad B = \lambda_1^2\cos^2{\varphi(u)} + \lambda_2^2\sin^2{\varphi(u)}, \quad D = \frac{\sin{\theta}}{c},
\end{align*}
then the two ODEs above could be rewritten as
\begin{align}\label{5.1}
    \varphi'^2(u) = c^2 + 2\cos{\theta}B -D^2B^2, \quad \varphi(0)=0,
\end{align}
and 
\begin{align}\label{5.2}
    f'(u) = DB, \quad f(0)=0.
\end{align}
After a straightforward computation, we obtain
\begin{align}\label{5.3}
    Bf''(u) = -2(\lambda_1^2-\lambda_2^2)f'(u)\varphi'(u)\sin{\varphi(u)}\cos{\varphi(u)},
\end{align}
and 
\begin{align}\label{5.4}
    B\varphi''(u) = (\lambda_1^2-\lambda_2^2)\left(f'^2(u) - \varphi'^2(u) +c^2 \right)\sin{\varphi(u)}\cos{\varphi(u)}.
\end{align}
It is easy to see that $\lambda_2^2 \leq B \leq \lambda_1^2$. Let us denote $\theta_{c}^{+} = \pi$ if $c>\sqrt{2}\lambda_1$ and $\theta_{c}^{+} = \arccos{\left(1-\frac{c^2}{\lambda_1^2}\right)}\in(0,\pi]$ if $0<c\leq \sqrt{2}\lambda_1$. We also define 
\begin{align*}
    \Omega = \{(c,\theta)\in \mathbb{R}^2\mid c>0, \theta\in(-\theta_{c}^{+}, \theta_{c}^{+})\}.
\end{align*}
\begin{lemma}\label{lem 5.1}
For any $(c,\theta)\in\Omega$ and $B\in[\lambda_2^2,\lambda_1^2]$, we always have 
\begin{align*}
    c^2 + 2\cos{\theta}B -D^2B^2 > 0.
\end{align*}
\end{lemma}
\begin{proof}
(1). If $\theta\in\pi\mathbb{Z}$, then $\sin{\theta}=0$. In addition, since $(c,\theta)\in\Omega$, we must have $\cos{\theta}=1$, whence 
\begin{align*}
    c^2 + 2\cos{\theta}B -D^2B^2 = c^2 +2B>0.
\end{align*}
(2). In the case when $\theta\notin\pi\mathbb{Z}$, we have $\frac{\sin^2{\theta}}{c^2}>0$ and
\begin{align*}
    c^2 + 2\cos{\theta}B -D^2B^2 = -D^2\left(B-\frac{c^2}{1-\cos{\theta}}\right)\left(B+\frac{c^2}{1+\cos{\theta}}\right).
\end{align*}
The condition $(c,\theta)\in\Omega$ implies that 
\begin{align*}
    \frac{c^2}{1-\cos{\theta}} > \lambda_1^2, \quad -\frac{c^2}{1+\cos{\theta}} < 0 < \lambda_2^2.
\end{align*}
Thus when $B\in[\lambda_2^2,\lambda_1^2]$, we always have 
\begin{align*}
    c^2 + 2\cos{\theta}B -D^2B^2 >0.
\end{align*}
\end{proof}

\begin{proposition}\label{prop 5.1}
Let $(c,\theta)\in\Omega$ and $\varphi$ be the solution of the ODE
\begin{align}\label{5.5}
    \varphi'(u) = \sqrt{c^2 + 2\cos{\theta}B -D^2B^2}, \quad \varphi(0)=0.
\end{align}
Then \\
(1). The function $\varphi$ is a well-defined increasing bijection from $\mathbb{R}$ to $\mathbb{R}$.\\
(2). The function $\varphi$ is an odd function.\\
(3). There exists a positive real number $U$ such that 
\begin{align*}
    \varphi(u+U) = \varphi(u) + \pi, \quad \forall u\in\mathbb{R}.
\end{align*}
(4). $\varphi(kU) = k\pi$ for all $k\in\mathbb{Z}$.\\
(5). $\varphi(k\frac{U}{2}) = k\frac{\pi}{2}$ for all $k\in2\mathbb{Z}+1$.
\end{proposition} 
\begin{proof}
(1). By the proof of Lemma \ref{lem 5.1}, we know that for $(c,\theta)\in\Omega$, there exist two real numbers $a_1, a_2>0$ such that $a_1\leq \varphi'(u)\leq a_2$. Thus the Cauchy-Lipschitz theorem tells us that $\varphi$ is well-defined on $\mathbb{R}$. Moreover, since $\lim\limits_{u\to+\infty} \varphi(u) = +\infty$ and $\lim\limits_{u\to-\infty} \varphi(u) = -\infty$, the function $\varphi$ is an increasing bijection from $\mathbb{R}$ to $\mathbb{R}$.\\
(2). The function defined by $\hat{\varphi}:=-\varphi(-u)$ satisfies ODE (\ref{5.5}) and the initial condition $\hat{\varphi}(0)=0$. This implies $\hat{\varphi}=\varphi$, whence $\varphi$ is odd.\\
(3). There exists a positive real number $U$ such that $\varphi(U)=\pi$ because $\varphi$ is an increasing function from $\mathbb{R}$ to $\mathbb{R}$. We define the function $\Tilde{\varphi}(u):=\varphi(u+U)-\pi$. One can check that $\Tilde{\varphi}$ also satisfies ODE (\ref{5.5}) and the initial condition $\Tilde{\varphi}(0)=0$, so $\Tilde{\varphi}=\varphi$.\\
(4). It is immediate.\\
(5). $\varphi(\frac{U}{2}) = \varphi(-\frac{U}{2} + U) = -\varphi(\frac{U}{2}) + \pi$, hence $\varphi(\frac{U}{2}) = \frac{\pi}{2}$. The rest follows easily from (3).
\end{proof}

\begin{proposition}\label{prop 5.2}
Let $(c,\theta)\in\Omega$ and $f$ be the function defined by ODE (\ref{5.2}), then\\
(1). $f$ is an odd function.\\
(2). $f(u+U) = f(u) + \pi$ for all $u\in\mathbb{R}$.\\
(3). $f(kU) = k\pi$ for all $k\in\mathbb{Z}$.\\
(4). $f(k\frac{U}{2}) = k\frac{\pi}{2}$ for all $k\in2\mathbb{Z}+1$.
\end{proposition}
\begin{proof}
(1). It is clear from (\ref{5.2}) that $f'$ is an even function. The condition $f(0)=0$ implies that $f$ is odd. The proof of the rest part is exactly the same with Proposition \ref{prop 5.1}.
\end{proof}

\begin{proposition}\label{prop 5.3}
The mapping $g$ satisfies
\begin{align*}
    g_{z\Bar{z}} = \frac{2\left[ \lambda_1^2(g + \Bar{g}) - \lambda_2^2(g - \Bar{g})\right]}{\lambda_1^2(g + \Bar{g})^2 - \lambda_2^2(g - \Bar{g})^2}g_{z}g_{\Bar{z}},
\end{align*}
and its Hopf differential is $Q = \frac{1}{8}e^{-i\theta}dz^2$.
\end{proposition}
\begin{proof}
We compute
\begin{align*}
    g_{z} = \frac{1}{2}e^{f(u)+cv}e^{i\varphi(u)}\left(f'(u)+i\varphi'(u)-ic\right),
\end{align*}
\begin{align*}
   g_{\Bar{z}} = \frac{1}{2}e^{f(u)+cv}e^{i\varphi(u)}\left(f'(u)+i\varphi'(u)+ic\right),
\end{align*}
\begin{align*}
    \Bar{g}_{z} = \frac{1}{2}e^{f(u)+cv}e^{-i\varphi(u)}\left(f'(u)-i\varphi'(u)-ic\right),
\end{align*}
\begin{align*}
   g_{z\Bar{z}} = \frac{1}{4}e^{f(u)+cv}e^{i\varphi(u)}\left[f'^2(u)-\varphi'^2(u)+f''(u)+c^2+i\varphi''(u)+2if'(u)\varphi'(u)\right]. 
\end{align*} 
Then one can check directly that the equation is satisfied. Together with (\ref{5.3}) and (\ref{5.4}), we can get the expression of $Q$.
\end{proof}

Now we want to use the Weierstrass representation again to construct a conformal minimal immersion  $X(u+iv) = (x_1, x_2, x_3)$ into $\widetilde{E(2)}$ with this prescribed Gauss map $g$. Firstly, we compute the $0$-potential which is \begin{align*}
    R(g) = -2ie^{2f(u)+2c}B.
\end{align*}
Then we get 
\begin{align*}
    \eta = \frac{4\Bar{g}g_{z}}{R(g)} = \frac{1}{B}\left((c-\varphi'(u)) + if'(u)\right).
\end{align*}
Secondly, let us write $X_{z} = \sum_{i=1}^{3}{x_i}_{z} \partial_{x_i} = \sum_{i=1}^{3}A_i E_i$, then Theorem \ref{thm 3.1} tells us that
\begin{align*}
    A_1 = & \frac{1}{2B}\left\{\left[(c-\varphi'(u))\sinh{A}\cos{\varphi(u)} +f'(u)\cosh{A}\sin{\varphi(u)}\right]\right.\\  
    &  \left. +i\left[f'(u)\sinh{A}\cos{\varphi(u)}-(c-\varphi'(u))\cosh{A}\sin{\varphi(u)}\right]\right\},
\end{align*}
\begin{align*}
    A_2 = & \frac{1}{2B}\left\{\left[(c-\varphi'(u))\sinh{A}\sin{\varphi(u)} -f'(u)\cosh{A}\cos{\varphi(u)}\right]\right.\\ 
    &\left. +i\left[f'(u)\sinh{A}\sin{\varphi(u)}+(c-\varphi'(u))\cosh{A}\cos{\varphi(u)}\right]\right\},
\end{align*}
\begin{align*}
    A_3 = \frac{1}{2B}\left[(c-\varphi'(u))+if'(u)\right].
\end{align*}
Comparing the two expressions of $X_z$, we obtain
\begin{equation}\label{5.6}
\left\{
\begin{aligned}
  & ~{x_1}_z = \frac{1}{\lambda_1}A_1\cos{x_3} - \frac{1}{\lambda_2}A_2\sin{x_3},\\
  & ~{x_2}_z = \frac{1}{\lambda_1}A_1\sin{x_3} + \frac{1}{\lambda_2}A_2\cos{x_3},\\
  & ~{x_3}_z = \lambda_1\lambda_2A_3.
\end{aligned}
\right.
\end{equation}
The formula for ${x_3}_z$ is equivalent to 
\begin{equation}\label{5.7}
\left\{
\begin{aligned}
  & ~\frac{\partial x_3}{\partial u} = \lambda_1\lambda_2\frac{D^2B-2\cos{\theta}}{c+\varphi'(u)} = \lambda_1\lambda_2\frac{c-\varphi'(u)}{B},\\
  & ~\frac{\partial x_3}{\partial v} = -\lambda_1\lambda_2D.
\end{aligned}
\right.
\end{equation}
Let us define 
\begin{align}\label{5.8}
    G(u) = \mathlarger{\int}_{0}^{u}\frac{D^2B-2\cos{\theta}}{c+\varphi'(s)}{\ud}s= \mathlarger{\int}_{0}^{u}\frac{c-\varphi'(s)}{B}{\ud}s,
\end{align}
then we have $G(0)=0$ and 
\begin{align}\label{5.9}
    x_3 = -\lambda_1\lambda_2Dv + \lambda_1\lambda_2G(u).
\end{align}

\begin{proposition}\label{prop 5.4}
Let $(c,\theta)\in\Omega$ and $G$ be the function defined by (\ref{5.8}), then\\
(1). $G$ is an odd function.\\
(2). $G(u+U) = G(u) + G(U)$ for all $u\in\mathbb{R}$.\\
(3). $G(kU) = kG(U)$ for all $k\in\mathbb{Z}$.\\
(4). $G(k\frac{U}{2}) = k\frac{G(U)}{2}$ for all $k\in2\mathbb{Z}+1$.
\end{proposition}
\begin{proof}
The method of proof is the same with the one of Proposition \ref{prop 5.2}.
\end{proof}

Let us set
\begin{equation}\label{5.10}
\left\{
\begin{aligned}
  & ~M_1 = c\cos{x_3}\cosh{A} - \lambda_1\lambda_2D\sin{x_3}\sinh{A},\\
  & ~M_2 = c\cos{x_3}\sinh{A} - \lambda_1\lambda_2D\sin{x_3}\cosh{A},\\
  & ~M_3 = c\sin{x_3}\sinh{A} + \lambda_1\lambda_2D\cos{x_3}\cosh{A},\\
  & ~M_4 = c\sin{x_3}\cosh{A} + \lambda_1\lambda_2D\cos{x_3}\sinh{A}.
\end{aligned}
\right.
\end{equation}
%\begin{align*}
   % M_1 = c\cos{x_3}\cosh{A} - \lambda_1\lambda_2D\sin{x_3}\sinh{A},
%\end{align*}
%\begin{align*}
 %   M_2 = c\cos{x_3}\sinh{A} - \lambda_1\lambda_2D\sin{x_3}\cosh{A},
%\end{align*}
%\begin{align*}
 %   M_3 = c\sin{x_3}\sinh{A} + \lambda_1\lambda_2D\cos{x_3}\cosh{A},
%\end{align*}
%\begin{align*}
 %   M_4 = c\sin{x_3}\cosh{A} + \lambda_1\lambda_2D\cos{x_3}\sinh{A}.
%\end{align*}
The solutions of equations (\ref{5.7}) are 
\begin{align*}
    x_1(u+iv) = & -\frac{1}{(c^2+\lambda_1^2\lambda_2^2D^2)B}\left[\frac{1}{\lambda_1}f'(u)\cos{\varphi(u)}M_1 - \frac{1}{\lambda_1}\left(c-\varphi'(u)\right)\sin{\varphi(u)}M_2\right.\\
    &\left. - \frac{1}{\lambda_2}\left(c-\varphi'(u)\right)\cos{\varphi(u)}M_3 - \frac{1}{\lambda_2}f'(u)\sin{\varphi(u)}M_4\right],
\end{align*}
\begin{align*}
    x_2(u+iv) = & -\frac{1}{(c^2+\lambda_1^2\lambda_2^2D^2)B}\left[\frac{1}{\lambda_1}f'(u)\cos{\varphi(u)}M_4 - \frac{1}{\lambda_1}\left(c-\varphi'(u)\right)\sin{\varphi(u)}M_3\right.\\
    &\left. + \frac{1}{\lambda_2}\left(c-\varphi'(u)\right)\cos{\varphi(u)}M_2 + \frac{1}{\lambda_2}f'(u)\sin{\varphi(u)}M_1\right].
\end{align*}
Consequently, we are able to summarise these results to the following theorem.

\begin{theorem}\label{thm 5.1}
Let $(c,\theta)\in\Omega$, $\varphi$ be the solution of ODE (\ref{5.5}) and $f$ be the solution of ODE (\ref{5.2}). Then the conformal minimal immersion 
\begin{align*}
  X(u+iv) = (x_1, x_2, x_3): \mathbb{C}\rightarrow\widetilde{E(2)}  
\end{align*}
with the Gauss map 
\begin{align*}
    g(z = u+iv) = e^{f(u)+cv}e^{i\varphi(u)},
\end{align*}
is given by
\begin{align*}
    x_1(u+iv) = & -\frac{1}{(c^2+\lambda_1^2\lambda_2^2D^2)B}\left[\frac{1}{\lambda_1}f'(u)\cos{\varphi(u)}M_1 - \frac{1}{\lambda_1}\left(c-\varphi'(u)\right)\sin{\varphi(u)}M_2\right.\\
    &\left. - \frac{1}{\lambda_2}\left(c-\varphi'(u)\right)\cos{\varphi(u)}M_3 - \frac{1}{\lambda_2}f'(u)\sin{\varphi(u)}M_4\right],
\end{align*}
\begin{align*}
    x_2(u+iv) = & -\frac{1}{(c^2+\lambda_1^2\lambda_2^2D^2)B}\left[\frac{1}{\lambda_1}f'(u)\cos{\varphi(u)}M_4 - \frac{1}{\lambda_1}\left(c-\varphi'(u)\right)\sin{\varphi(u)}M_3\right.\\
    &\left. + \frac{1}{\lambda_2}\left(c-\varphi'(u)\right)\cos{\varphi(u)}M_2 + \frac{1}{\lambda_2}f'(u)\sin{\varphi(u)}M_1\right].
\end{align*}
and
\begin{align*}
    x_3(u+iv) = -\lambda_1\lambda_2Dv + \lambda_1\lambda_2G(u),
\end{align*}
where $G$ and $M_1,M_2,M_3,M_4$ are defined as in (\ref{5.8}) and (\ref{5.10}).
\end{theorem}

Given $(c,\theta)\in\Omega$ and $U>0$ as in Proposition \ref{prop 5.1}, we define a function 
\begin{align}\label{5.11}
    H(c,\theta) & = Df(U) +cG(U)
     = \int_{0}^{U}D^2B{\ud}u + \int_{0}^{U}\frac{cD^2B^2-2c\cos{\theta}}{c+\varphi'(u)}{\ud}u.
\end{align}
Let us write $x=\cos{\varphi(u)}$ and $P(x) =  c^2 + 2\cos{\theta}B -D^2B^2$. Since $0\leq u\leq U$, we have $0\leq \varphi(u)\leq \pi$ and $-1\leq x \leq1$. Therefore, we obtain ${\ud}u = -\frac{{\ud}x}{\sqrt{(1-x^2)P(x)}}$, thus the function $H(c,\theta)$ also takes the form
\begin{align}\label{5.12}
    H(c,\theta) = \mathlarger{\int}_{-1}^{1}\frac{D^2B\left(2c+\sqrt{P(x)}\right)-2c\cos{\theta}}{\left(c+\sqrt{P(x)}\right)\sqrt{(1-x^2)P(x)}}{\ud}x.
\end{align}

\begin{lemma}\label{lem 5.2}
For any $c>0$, there exists a unique $\widetilde{\theta_{c}}\in(0,\frac{\pi}{2})\cap(0,\theta_{c}^+)$ such that
\begin{align*}
    H(c,\widetilde{\theta_{c}})=0.
\end{align*}
\end{lemma}
\begin{proof}
It is easy to see that $H(c,\theta)$ is a continuous function on $\Omega$. This function has the property that 
\begin{align*}
    H(c,0) = \mathlarger{\int}_{-1}^{1}\frac{-2c}{\left(c+\sqrt{P(x)}\right)\sqrt{(1-x^2)P(x)}}{\ud}x < 0.
\end{align*}
(1). If $c>\sqrt{2}\lambda_1$, then $\theta_{c}^+=\pi$ and 
\begin{align*}
    H(c,\frac{\pi}{2}) = \mathlarger{\int}_{-1}^{1}\frac{B\left(2c+\sqrt{P(x)}\right)}{\left(c+\sqrt{P(x)}\right)\sqrt{(1-x^2)P(x)}}{\ud}x > 0.
\end{align*}
Hence there exists $\widetilde{\theta_{c}}\in(0,\frac{\pi}{2})\cap(0,\theta_{c}^+)$ such that $H(c,\widetilde{\theta_{c}})=0$.\\
(2). If $0<c\leq\sqrt{2}\lambda_1$, then $\theta_{c}^{+} = \arccos{\left(1-\frac{c^2}{\lambda_1^2}\right)}$.

(i). When $\lambda_1 < c \leq \sqrt{2}\lambda_1$, we have $\theta_{c}^{+}>\frac{\pi}{2}$. Then $H(c,\frac{\pi}{2})>0$ still holds, whence there is a $\widetilde{\theta_{c}}\in(0,\frac{\pi}{2})\cap(0,\theta_{c}^+)$ such that $H(c,\widetilde{\theta_{c}})=0$.

(ii). The only situation remained to be discussed is $0<c\leq\lambda_1$. In this case, we have $\theta_{c}^{+}\leq\frac{\pi}{2}$ and 
\begin{align*}
    \cos{\theta_{c}^{+}} = 1-\frac{c^2}{\lambda_1^2},\quad \sin^2{\theta_{c}^{+}} = 2\frac{c^2}{\lambda_1^2}-\frac{c^4}{\lambda_1^4}.
\end{align*}
In addition, we also get 
\begin{align*}
    \lim\limits_{\theta\to\theta_{c}^{+}} D^2  = \frac{2}{\lambda_1^2}-\frac{c^2}{\lambda_1^4}, \quad \lim\limits_{\theta\to\theta_{c}^{+}}\frac{c^2}{1-\cos{\theta}} = \lambda_1^2, \quad \lim\limits_{\theta\to\theta_{c}^{+}}-\frac{c^2}{1+\cos{\theta}}= -\frac{c^2}{2-\frac{c^2}{\lambda_1^2}}<0.
\end{align*}
Now we are going to prove $\lim\limits_{\theta\to\theta_{c}^{+}}H(c,\theta)>0$. Let us define two functions $L_1(c,\theta)$ and $L_2(c,\theta)$ by
\begin{align*}
    L_1(c,\theta) = \mathlarger{\int}_{-1}^{1}-\frac{2cD^2(\lambda_1^2-\lambda_2^2)}{c+\sqrt{P(x)}}\sqrt{\frac{1-x^2}{P(x)}}{\ud}x,
\end{align*}
\begin{align*}
    L_2(c,\theta) = \mathlarger{\int}_{-1}^{1}\frac{2c\lambda_1^2D^2 - 2c\cos{\theta} + D^2B\sqrt{P(x)}}{\left(c+\sqrt{P(x)}\right)\sqrt{(1-x^2)P(x)}}{\ud}x.
\end{align*}
Then $H(c,\theta) = L_1(c,\theta) + L_2(c,\theta)$.

From the estimation 
\begin{equation*}
\begin{aligned}
&\left|\frac{1-x^{2}}{P(x)}-\frac{1}{\left(\frac{\lambda_{1}^{2}+\lambda_{2}^{2}}{\lambda_{1}^{2}}\right)\left[\left(2-\frac{c^{2}}{\lambda_{1}^{2}}\right) B+c^{2}\right]}\right| \\\notag
\leq & \frac{D^{2}\left(\lambda_{1}^{2}-\lambda_{2}^{2}\right)^{2}+2\left(\frac{\lambda_{1}^{2}+\lambda_{2}^{2}}{\lambda_{1}^{2}}\right)\left(\lambda_{1}^{2}-\lambda_{2}^{2}\right)\left(2-\frac{c^{2}}{\lambda_{1}^{2}}\right)+2 c^{2}\left(\frac{2_{1}^{2}+\lambda_{2}^{2}}{\lambda_{1}^{2}}\right)}{D^{2}\left(\frac{\lambda_{1}^{2}+\lambda_{2}^{2}}{\lambda_{2}^{2}}\right)\left(\frac{c^{2}}{1-\cos \theta}-\lambda_{1}^{2}\right)\left(\lambda_{2}^{2}+\frac{c^{2}}{1+\cos \theta}\right)\left[\left(2-\frac{c^{2}}{\lambda_{1}^{2}}\right) \lambda_{2}^{2}+c^{2}\right]}\\
&  + \frac{D^{2} \lambda_{2}^{2}\left(2 \lambda_{1}^{2}-\lambda_{2}^{2}\right)+2 \cos \theta \lambda_{1}^{2}+c^{2}}{D^{2}\left(\frac{\lambda_{1}^{2}+\lambda_{2}^{2}}{\lambda_{2}^{2}}\right)\left(\frac{c^{2}}{1-\cos \theta}-\lambda_{1}^{2}\right)\left(\lambda_{2}^{2}+\frac{c^{2}}{1+\cos \theta}\right)\left[\left(2-\frac{c^{2}}{\lambda_{1}^{2}}\right) \lambda_{2}^{2}+c^{2}\right]},
\end{aligned}
\end{equation*}
we know there is a constant $M_{c,\theta}>0$ such that $\left|\frac{1-x^2}{P(x)}\right|\leq M_{c,\theta}$. Moreover, for each fixed $c$, when $\theta\rightarrow\theta_{c}^{+}$, this constant $M_{c,\theta}$ tends to a finite limit $M_c>0$. We may draw the conclusion that the function $\frac{1-x^2}{P(x)}$ is uniformly bounded when $\theta\rightarrow\theta_{c}^{+}$. Therefore, we obtain
\begin{align*}
    |L_1(c,\theta)|\leq 2D^2(\lambda_1^2-\lambda_2^2)\mathlarger{\int}_{-1}^{1}\sqrt{\frac{1-x^2}{P(x)}}{\ud}x \leq 4D^2(\lambda_1^2-\lambda_2^2)\sqrt{M_{c,\theta}}.
\end{align*}
This implies 
\begin{align*}
    \lim\limits_{\theta\to\theta_{c}^{+}}|L_1(c,\theta)|& \leq \lim\limits_{\theta\to\theta_{c}^{+}}4D^2(\lambda_1^2-\lambda_2^2)\sqrt{M_{c,\theta}}\\
    & = 4\left(\frac{2}{\lambda_1^2}-\frac{c^2}{\lambda_1^4}\right)(\lambda_1^2-\lambda_2^2)\sqrt{M_c}<+\infty.
\end{align*}
Thus the function $L_1(c,\theta)$ is bounded when $\theta\rightarrow\theta_{c}^{+}$.

We must also study the function $L_2(c,\theta)$. It is easy to see that $2cD^2\lambda_1^2-2c\cos{\theta}\rightarrow2c$ when $\theta\rightarrow\theta_{c}^{+}$. Hence for $\theta$ sufficiently close to $\theta_{c}^{+}$, we have 
\begin{align*}
    L_2(c,\theta)\geq\mathlarger{\int}_{-1}^{1}\frac{c}{\left(c+\sqrt{P(x)}\right)\sqrt{(1-x^2)P(x)}}{\ud}x.
\end{align*}
In addition, we notice that 
\begin{align*}
    &\sqrt{D^2\left(B+\frac{c^2}{1+\cos{\theta}}\right)}\left(c+\sqrt{P(x)}\right)\\
     & \leq \frac{\sqrt{2}}{\lambda_1}\sqrt{\lambda_1^2+\frac{c^2}{1+\cos{\theta}}}\left(c+\frac{\sqrt{2}}{\lambda_1}\sqrt{\lambda_1^2+\frac{c^2}{1+\cos{\theta}}}\sqrt{\frac{c^2}{1-\cos{\theta}}-\lambda_2^2}\right),
\end{align*}
thus for $\theta$ close enough to $\theta_{c}^{+}$, there is a positive constant $b_c$ such that 
\begin{align*}
    L_2(c,\theta)& \geq\mathlarger{\int}_{-1}^{1}\frac{c}{\left(c+\sqrt{P(x)}\right)\sqrt{(1-x^2)P(x)}}{\ud}x\\
    &\geq\mathlarger{\int}_{-1}^{1}\frac{b_c}{\sqrt{(1-x^2)\left(\frac{c^2}{1-\cos{\theta}}-B\right)}}{\ud}x.
\end{align*}
Consequently, 
\begin{align*}
     \lim\limits_{\theta\to\theta_{c}^{+}}L_2(c,\theta) &\geq \lim\limits_{\theta\to\theta_{c}^{+}}\mathlarger{\int}_{-1}^{1}\frac{b_c}{\sqrt{(1-x^2)\left(\frac{c^2}{1-\cos{\theta}}-B\right)}}{\ud}x\\
     & = \frac{b_c}{\sqrt{\lambda_1^2+\lambda_2^2}}\int_{-1}^{1}\frac{{\mathrm{d}}x}{1-x^2} = +\infty.
\end{align*}
This means $\lim\limits_{\theta\to\theta_{c}^{+}}H(c,\theta) = +\infty$, thus there exists a $\widetilde{\theta_{c}}\in(0,\frac{\pi}{2})\cap(0,\theta_{c}^+)$ such that $H(c,\widetilde{\theta_{c}})=0$.

The last task is to show that $\widetilde{\theta_{c}}\in(0,\frac{\pi}{2})\cap(0,\theta_{c}^+)$ is unique. In fact, we have
\begin{align*}
    H(c,\theta) & = \mathlarger{\int}_{-1}^{1}\frac{D^2B\left(2c+\sqrt{P(x)}\right)-2c\cos{\theta}}{\left(c+\sqrt{P(x)}\right)\sqrt{(1-x^2)P(x)}}{\ud}x\\
    & = \mathlarger{\int}_{-1}^{1}\frac{D^2B^2+c^2-c\sqrt{P(x)}}{B\sqrt{(1-x^2)P(x)}}{\ud}x.
\end{align*}
When $\theta\in(0,\frac{\pi}{2})\cap(0,\theta_{c}^+)$, we may compute the derivative 
\begin{equation*}
\begin{aligned}
\frac{\partial}{\partial \theta} H(c,\theta) &=\mathlarger{\int}_{-1}^{1} \frac{\partial}{\partial \theta}\left(\frac{D^{2} B^{2}+c^{2}-c \sqrt{P(x)}}{B\sqrt{\left(1-x^{2}\right) P(x)} }\right) {\ud} x \\
&=\mathlarger{\int}_{-1}^{1} \frac{2 \sin \theta \cos \theta B P(x)+\left(D^{2} B^{2}+c^{2}\right)\left(c^{2} \sin \theta+B \sin \theta \cos \theta\right)}{c^{2}\sqrt{1-x^{2}} P(x)^{\frac{3}{2}}} {\ud} x\\
& >0.
\end{aligned}
\end{equation*}
This implies that the function $\psi(\theta):=H(c,\theta)$ is increasing on $(0,\frac{\pi}{2})\cap(0,\theta_{c}^+)$, thus $\widetilde{\theta_{c}}$ is unique.
\end{proof}

\begin{lemma}\label{lem 5.3}
Let $c>0$ and $\widetilde{\theta_{c}}\in(0,\frac{\pi}{2})\cap(0,\theta_{c}^+)$ be defined as in Lemma \ref{lem 5.2}, then 
\begin{align*}
    \lim\limits_{c\to +\infty}\widetilde{\theta_{c}} = \frac{\pi}{2}.
\end{align*}
\end{lemma}
\begin{proof}
We fix $\theta=\widetilde{\theta_{c}}$, then $D_c = \frac{\sin{\widetilde{\theta_{c}}}}{c}\leq\frac{1}{c}$. When $c\geq2\lambda_1$, we can estimate 
\begin{align*}
    \left|P(x)-c^2\right| = \left|(2\cos{\widetilde{\theta_{c}}} - D_{c}^2B)B\right|\leq\lambda_1^2\left(2+\frac{\lambda_1^2}{c^2}\right)\leq\frac{9}{4}\lambda_1^2.
\end{align*}
This gives
\begin{align*}
    0<\frac{7}{4}\lambda_1^2\leq c^2-\frac{9}{4}\lambda_1^2\leq P(x)\leq c^2+\frac{9}{4}\lambda_1^2,
\end{align*}
thus 
\begin{align*}
    \sqrt{c^2-\frac{9}{4}\lambda_1^2}\leq\sqrt{P(x)}\leq\sqrt{c^2+\frac{9}{4}\lambda_1^2}.
\end{align*}
Let us denote 
\begin{align*}
    I_1(c) = \mathlarger{\int}_{-1}^{1}\frac{2cD_{c}^2B}{\left(c+\sqrt{P(x)}\right)\sqrt{(1-x^2)P(x)}}{\ud}x,
\end{align*}
\begin{align*}
    I_2(c) = \mathlarger{\int}_{-1}^{1}\frac{c^2}{\left(c+\sqrt{P(x)}\right)\sqrt{(1-x^2)P(x)}}{\ud}x,
\end{align*}
\begin{align*}
    I_3(c) = \mathlarger{\int}_{-1}^{1}\frac{cD_{c}^2B}{\left(c+\sqrt{P(x)}\right)\sqrt{(1-x^2}}{\ud}x.
\end{align*}
On the one hand, we obtain from Lemma \ref{lem 5.2} that 
\begin{align*}
    0 = cH(c,\widetilde{\theta_{c}}) = I_1(c) - 2\cos{\widetilde{\theta_{c}}}I_2(c) + I_3(c).
\end{align*}
One the other hand, when $c$ tends to $\infty$, we have 
\begin{align*}
    \lim\limits_{c\to +\infty}I_1(c) = 0,\quad \lim\limits_{c\to +\infty}I_2(c) =\frac{(\lambda_1^2+\lambda_2^2)\pi}{4},\quad \lim\limits_{c\to +\infty}I_3(c) = 0.
\end{align*}
This means 
\begin{align*}
    0 = \lim\limits_{c\to +\infty}cH(c,\widetilde{\theta_{c}}) = -2\left(\lim\limits_{c\to +\infty}\cos{\widetilde{\theta_{c}}}\right)\frac{(\lambda_1^2+\lambda_2^2)\pi}{4},
\end{align*}
whence $\lim\limits_{c\to +\infty}\cos{\widetilde{\theta_{c}}}=0$ and $\lim\limits_{c\to +\infty}\widetilde{\theta_{c}}=\frac{\pi}{2}$. 
\end{proof}

\begin{theorem}\label{thm 5.2}
Let $c>0$ and $\widetilde{\theta_{c}}\in(0,\frac{\pi}{2})\cap(0,\theta_{c}^+)$ be defined as in Lemma \ref{lem 5.2}. Then there exists a $Z\in\mathbb{C}\setminus\{0\}$ such that the corresponding minimal immersion $X:\mathbb{C}\rightarrow\widetilde{E(2)}$ given by Theorem \ref{thm 5.1} satisfies 
\begin{align*}
    X(z + Z) = X(z), \quad \forall z\in\mathbb{C}.
\end{align*}
\end{theorem}
\begin{proof}
Let us take $Z=2U-i\frac{2f(U)}{c}$. Since $U>0$ and 
\begin{align*}
    f(U) = \int_{0}^{U} \frac{\sin{\widetilde{\theta_{c}}}}{c}\left(\lambda_1^2\cos^2{\varphi(u)} + \lambda_2^2\sin^2{\varphi(u)}\right){\ud}u>0,
\end{align*}
we have $Z\in\mathbb{C}\setminus\{0\}$. Now we need to verify that the coordinate functions $x_1, x_2, x_3$ are all $Z$-periodic.

(i). For $x_3$, we utilise formula (\ref{5.9}), which gives
\begin{align*}
    x_3(z+Z) & = x_3\left((u+2U) + i\left(v-\frac{2f(U)}{c}\right)\right)\\
    & = -\lambda_1\lambda_2D\left(v-\frac{2f(U)}{c}\right) + \lambda_1\lambda_2G(u+2U)\\
    & = -\lambda_1\lambda_2Dv + \lambda_1\lambda_2G(u) +\frac{2\lambda_1\lambda_2}{c}H(c,\widetilde{\theta_{c}})\\
    & = x_3(z).
\end{align*}

(ii). As for $x_1$ and $x_2$, by Proposition \ref{prop 5.1} and Proposition \ref{prop 5.2}, we can see that it is enough to show the $Z$-periodicity of $A(u+iv)=f(u)+cv$. We may compute 
\begin{align*}
    A(z+Z) & = A\left((u+2U) + i\left(v-\frac{2f(U)}{c}\right)\right)\\
    & = f(u+2U) +c\left(v-\frac{2f(U)}{c}\right)\\
    & = f(u)+cv\\
    & = A(z).
\end{align*}
This completes the proof.
\end{proof}

\begin{remark}
This kind of period problem does not appear in Desmonts' paper \cite{desmonts2015constructions}. In his case, some important functions which are similar to our $f$ and $G$ already possess some periodic properties, thus the immersion itself is automatically periodic.
\end{remark}\label{rmk 5.1}

\begin{definition}\label{def 5.1}
   Let $c>0$ and $\widetilde{\theta_{c}}\in(0,\frac{\pi}{2})\cap(0,\theta_{c}^+)$ be defined as in Lemma \ref{lem 5.2}. The minimal surface given by $X:\mathbb{C}\rightarrow\widetilde{E(2)}$ in Theorem \ref{thm 5.1} with $\theta = \widetilde{\theta_{c}}$ is called a catenoid of parameter $c$ and it will be denoted by $\mathfrak{C}_c$.
\end{definition}

The metric induced by the immersion $X:\mathbb{C}\rightarrow\widetilde{E(2)}$ on the catenoid $\mathfrak{C}_c$ is given by
\begin{align}\label{5.13}
    \mathrm{d}s^2 = \rho^2|\mathrm{d}z|^2 := \frac{f'^2(u) + \left(c-\varphi'(u)\right)^2}{B^2}\cosh^2{\left(f(u)+cv\right)}|\mathrm{d}z|^2. 
\end{align}

\begin{remark}\label{rmk 5.2}
The catenoid $\mathfrak{C}_c$ has some symmetries. One may check that
\begin{equation*}
\left\{
\begin{aligned}
  & ~x_1(z+\frac{Z}{2}) = -x_1(z),\\
  & ~x_2(z+\frac{Z}{2}) = -x_2(z),\\
  & ~x_3(z+\frac{Z}{2}) = x_3(z).
\end{aligned}
\right.
\end{equation*}
It means the catenoid $\mathfrak{C}_c$ is symmetric by rotation of angle $\pi$ around the $x_3$-axis. In a similar way, we get
 \begin{equation*}
\left\{
\begin{aligned}
  & ~x_1(-z) = x_1(z),\\
  & ~x_2(-z) = -x_2(z),\\
  & ~x_3(-z) = -x_3(z).
\end{aligned}
\right.
\end{equation*}
Thus $\mathfrak{C}_c$ is symmetric by rotation of angle $\pi$ around the $x_1$-axis. Therefore, $\mathfrak{C}_c$ is also symmetric by rotation of angle $\pi$ around the $x_2$-axis.
\end{remark}

Now we would like to study some other interesting geometric properties of the catenoid $\mathfrak{C}_c$.

\begin{proposition}\label{prop 5.5}
For every $\mu\in\mathbb{R}$, the intersection of the catenoid $\mathfrak{C}_c$ with the plane  $\{x_3 = \lambda_1\lambda_2\mu\}$ is a non-empty closed embedded convex curve.
\end{proposition}
\begin{proof}
(1). The intersection is not empty because $x_3\left(0-i\frac{\mu}{D}\right)=\lambda_1\lambda_2\mu$. \\
(2). Let us denote the set of intersection by $\gamma$. The condition 
\begin{align*}
    x_3(u+iv) = -\lambda_1\lambda_2Dv + \lambda_1\lambda_2G(u) = \lambda_1\lambda_2\mu
\end{align*}
implies $v=\frac{G(u)-\mu}{D}$. By setting 
\begin{equation}\label{5.14}
\left\{
\begin{aligned}
  & ~J_1 = f'(u)\cos{\varphi(u)}\cosh{A} - \left(c-\varphi'(u)\right)\sin{\varphi(u)}\sinh{A},\\
  & ~J_2 = f'(u)\cos{\varphi(u)}\sinh{A} - \left(c-\varphi'(u)\right)\sin{\varphi(u)}\cosh{A},\\
  & ~J_3 = f'(u)\sin{\varphi(u)}\cosh{A} + \left(c-\varphi'(u)\right)\cos{\varphi(u)}\sinh{A},\\
  & ~J_4 = f'(u)\sin{\varphi(u)}\sinh{A} + \left(c-\varphi'(u)\right)\cos{\varphi(u)}\cosh{A},
\end{aligned}
\right.
\end{equation}
we can determine the expression of $\gamma$ to be 
\begin{align*}
    \gamma(u) & = X\left(u+i\left(\frac{G(u)-\mu}{D}\right)\right)\\
    & =\begin{pmatrix}
        -\frac{1}{\left(c^2+\lambda_1^2\lambda_2^2D^2\right)B}\left(\frac{c}{\lambda_1}\cos{x_3}J_1 - \lambda_2D\sin{x_3}J_2 - \frac{c}{\lambda_2}\sin{x_3}J_3 -\lambda_1D\cos{x_3}J_4\right)\\
        -\frac{1}{\left(c^2+\lambda_1^2\lambda_2^2D^2\right)B}\left(\frac{c}{\lambda_1}\sin{x_3}J_1 + \lambda_2D\cos{x_3}J_2 + \frac{c}{\lambda_2}\cos{x_3}J_3 -\lambda_1D\sin{x_3}J_4\right)\\
        x_3
        \end{pmatrix}.
\end{align*}
It is easy to see that $\gamma$ is a closed curve because $\gamma(u+2U)=\gamma(u)$ for all $u\in\mathbb{R}$.

We still need to prove that $\gamma$ is smooth and convex. We notice that the left multiplication of $(0,0,-x_3)$ with $x_3=\lambda_1\lambda_2\mu$ is an isometry of $\widetilde{E(2)}$ which preserves the smoothness and convexity of $\gamma$. Therefore, we are going to study the curve 
\begin{align*}
    \widetilde{\gamma}(u) & := (0,0,-x_3)*\gamma(u)\\
    & =\begin{pmatrix}
        -\frac{1}{\left(c^2+\lambda_1^2\lambda_2^2D^2\right)B}\left(\frac{c}{\lambda_1}J_1 - \lambda_1DJ_4\right)\\
        -\frac{1}{\left(c^2+\lambda_1^2\lambda_2^2D^2\right)B}\left(\frac{c}{\lambda_2}J_3 + \lambda_2DJ_2\right)\\
        0
        \end{pmatrix}\\
    & = (\widetilde{x_1}(u), \widetilde{x_2}(u),0).    
\end{align*}
(3). A direct computation shows
\begin{equation}\label{5.15}
\left\{
\begin{aligned}
  & ~\widetilde{x_1}'(u) = \frac{f'^2(u) + \left(c-\varphi'(u)\right)^2}{\lambda_1DB^2}\sin{\varphi(u)}\cosh{A},\\
  & ~\widetilde{x_2}'(u) = -\frac{f'^2(u) + \left(c-\varphi'(u)\right)^2}{\lambda_2DB^2}\cos{\varphi(u)}\cosh{A}.
\end{aligned}
\right.
\end{equation}
As a result, we get 
\begin{align*}
   & \widetilde{x_1}'^2(u) + \widetilde{x_2}'^2(u) \\
   = & ~\frac{\left[f'^2(u) + \left(c-\varphi'(u)\right)^2\right]^2}{D^2B^4}\left(\frac{1}{\lambda_1^2}\sin^2{\varphi(u)} + \frac{1}{\lambda_2^2}\cos^2{\varphi(u)}\right)\cosh^2{A}>0.
\end{align*}
This means $\widetilde{\gamma}$ is a smooth curve, thus $\gamma$ is smooth as well.\\
(4). Since $\widetilde{x_1}(u+U)=-\widetilde{x_1}(u)$ and $\widetilde{x_2}(u+U)=-\widetilde{x_2}(u)$, the curve $\widetilde{\gamma}$ is also invariant by rotation of angle $\pi$ around the $x_3$-axis. Consequently, we may first consider the half of the curve $\widetilde{\gamma}$ corresponding to $u\in(-\frac{U}{2},\frac{U}{2})$. On this interval, we have $\cos{\varphi(u)}>0$, thus $\widetilde{x_2}'(u)<0$. This fact tells us that the function $\widetilde{x_2}(u)$ is a decreasing injection. In addition, thanks to Proposition \ref{prop 5.1}, we obtain that
\begin{align*}
    \frac{{\ud} \widetilde{x_1}}{{\ud} \widetilde{x_2}} = -\frac{\lambda_2}{\lambda_1}\tan{\varphi(u)}
\end{align*}
is a decreasing function of $u$, thus $\frac{{\ud}^2\widetilde{x_1}}{{{\ud}\widetilde{x_2}}^2}>0$. Therefore, the half of $\widetilde{\gamma}$ corresponding to $u\in(-\frac{U}{2},\frac{U}{2})$ is convex and embedded. Then by the symmetries mentioned above, the whole curve $\widetilde{\gamma}$ is convex and embedded. Moreover, the intersection set $\gamma$ is also an embedded convex curve.     
\end{proof}

\begin{remark}\label{rmk 5.3}
From Remark \ref{rmk 5.2} as well as the proof of Proposition \ref{prop 5.5}, we can see that the $x_3$-axis is actually contained in the interior of the catenoid $\mathfrak{C}_c$.
\end{remark}

\begin{proposition}\label{prop 5.6}
For every $c>0$, the catenoid $\mathfrak{C}_c$ is properly embedded.
\end{proposition}
\begin{proof}
We already know from Proposition \ref{prop 5.5} that $\mathfrak{C}_c$ is embedded. For each diverging path $\alpha$ on $\mathfrak{C}_c$, the function $A$ is also diverging, thus $x_3$ is diverging. This proves that $\mathfrak{C}_c$ is properly embedded.  
\end{proof}

\begin{proposition}\label{prop 5.7}
For each $c>0$, the catenoid $\mathfrak{C}_c$ is conformally equivalent to $\mathbb{C}\setminus\{0\}$.
\end{proposition}
\begin{proof}
By Theorem \ref{thm 5.2}, the immersion $X:\mathbb{C}\rightarrow\widetilde{E(2)}$ induces a conformal bijective parametrization of $\mathfrak{C}_c$ by $\widetilde{X}:\mathbb{C}/(Z\mathbb{Z})\rightarrow\mathfrak{C}_c$.
\end{proof}

In the end of this section, we would like to study some curvature properties of the catenoid $\mathfrak{C}_c$. Taking advantage of the expression (\ref{5.13}) for the induced metric $ds^2=\rho^2|dz|^2$, we can compute
\begin{equation*}
\begin{aligned}
\Delta \ln \rho 
&= \frac{c^{2}+f'^2(u)}{\cosh ^{2}A}+\frac{\left(\lambda_{1}^{2}-\lambda_{2}^{2}\right) \sin ^{2}\left(2 \varphi(u)\right)\left[f'^2(u)+\left(\varphi^{\prime}(u)-c\right)^{2}\right]}{4B^2} \\
& \quad+\frac{\left(\lambda_{1}^{2}-\lambda_{2}^{2}\right) \varphi^{\prime}(u)\left[\left(\varphi^{\prime}(u)-c\right) \cos \left(2 \varphi(u)\right) - f'(u)\tanh {A}\sin \left(2 \varphi(u)\right)\right]}{B},
\end{aligned}
\end{equation*}
hence the Gauss curvature is
\begin{equation*}
\begin{aligned}
\mathcal{K} &=-\frac{B^2\left(c^{2}+f'^2(u)\right)}{\cosh ^{4}{A}\left[f'^2(u)+\left(\varphi^{\prime}(u)-c\right)^{2}\right]}- \frac{\left(\lambda_{1}^{2} - \lambda_{2}^{2}\right)^{2} \sin ^{2}\left(2 \varphi(u)\right)}{4 \cosh ^{2}{A}} \\
&\quad-\frac{\left(\lambda_{1}^{2}-\lambda_{2}^{2}\right)B \varphi^{\prime}(u)\left[\left(\varphi^{\prime}(u)-c\right) \cos \left(2 \varphi(u)\right) - f'(u)\tanh {A}\sin \left(2 \varphi(u)\right)\right]}{\cosh ^{2}{A}\left[f'^2(u)+\left(\varphi^{\prime}(u)-c\right)^{2}\right]} .
\end{aligned}
\end{equation*}

A natural question is to determine the total absolute curvature of the catenoid $\mathfrak{C}_c$. Let us use ${\ud}A$ to represent the area element of $\mathfrak{C}_c$, then 
\begin{align*}
    |\mathcal{K}|{\ud}A = \left|\Delta \ln \rho\right|{\ud}u \mathrm{d}v.
\end{align*}
(1). If $\lambda_1=\lambda_2=1$, then we may compute 
\begin{align*}
    \int_{\mathfrak{C}_c}|\mathcal{K}|{\ud}A = \int_{-\infty}^{\infty}\int_{-U}^{U}\frac{c^2+D^2}{\cosh^2{(Du+cv)}}{\ud}u \mathrm{d}v = \frac{4(c^2+D^2)U}{c} < +\infty.
\end{align*}
Actually in this case we have
\begin{align*}
    U = \frac{\pi}{\sqrt{c^2 + 2\cos{\widetilde{\theta_{c}}}-D^2}}
\end{align*}
and 
\begin{align*}
    H(c,\widetilde{\theta_{c}}) = \frac{D^2+c^2-c\sqrt{c^2 + 2\cos{\widetilde{\theta_{c}}}-D^2}}{\sqrt{c^2 + 2\cos{\widetilde{\theta_{c}}}-D^2}}\pi = 0.
\end{align*}
Hence the finite total absolute curvature is $4\pi$. This corresponds to the classic catenoids in $\mathbb{R}^3$.\\
(2). In the case when $\lambda_1 > \lambda_2 > 0$, we choose $u_0\in\mathbb{R}$ satisfying $\varphi(u_0)=\frac{\pi}{4}$. For a sufficiently small $\epsilon>0$, let us consider a region $N\subset\mathbb{C}$ in which $u\in(u_0-\epsilon,u_0+\epsilon)$ and $f(u)+cv\ll0$. Then the total absolute curvature
\begin{align*}
    \int_{\mathfrak{C}_c}|\mathcal{K}|{\ud}A = \int_{-\infty}^{\infty}\int_{-U}^{U}\left|\Delta \ln \rho\right|{\ud}u \mathrm{d}v\geq \int_{N} \left|\Delta \ln \rho\right|{\ud}u \mathrm{d}v = +\infty.
\end{align*}

%\begin{figure}[htpb]
%\centering\includegraphics[width=0.7\linewidth]{./figure/2.pdf}\vspace{-.5cm}
%\centering\small{345566666}
%\caption{1234355.}\label{fg:1}
%\end{figure}

%\begin{figure}[htpb]
%\centering\includegraphics[width=0.7\linewidth]{./figure/4.pdf}\vspace{-.5cm}
%\centering\small{345566666}
%\caption{1234355.}\label{fg:1}
%\end{figure}

%\begin{figure}[!htp]
%\begin{figure}[H]
  %\centering
  %\begin{subfigure}{0.49\textwidth}
    %\includegraphics[width=0.98\textwidth]{./figure/12.pdf}\\[-0.1cm]
    %\caption{Case 1 }
  %\end{subfigure}
    %\begin{subfigure}{0.49\textwidth}
    %\includegraphics[width=0.98\textwidth]{./figure/11.pdf}\\[-0.1cm]
    %\caption{Case 2 }
  %\end{subfigure}\\[6mm]
  %\begin{subfigure}{0.49\textwidth}
    %\includegraphics[width=0.98\textwidth]{./figure/4.pdf}\\[-0.1cm]
    %\caption{Case 3 }
  %\end{subfigure}
  %\begin{subfigure}{0.49\textwidth}
    %\includegraphics[width=0.98\textwidth]{./figure/9.pdf}\\[-0.1cm]
    %\caption{Case 4 }
  %\end{subfigure}\\[6mm]
  %\caption{newnewnew.}
  %\label{fig:simulation2}
%\end{figure}

\begin{figure}[H]
  \centering
\begin{subfigure}{0.49\textwidth}
    \includegraphics[width=0.95\textwidth]{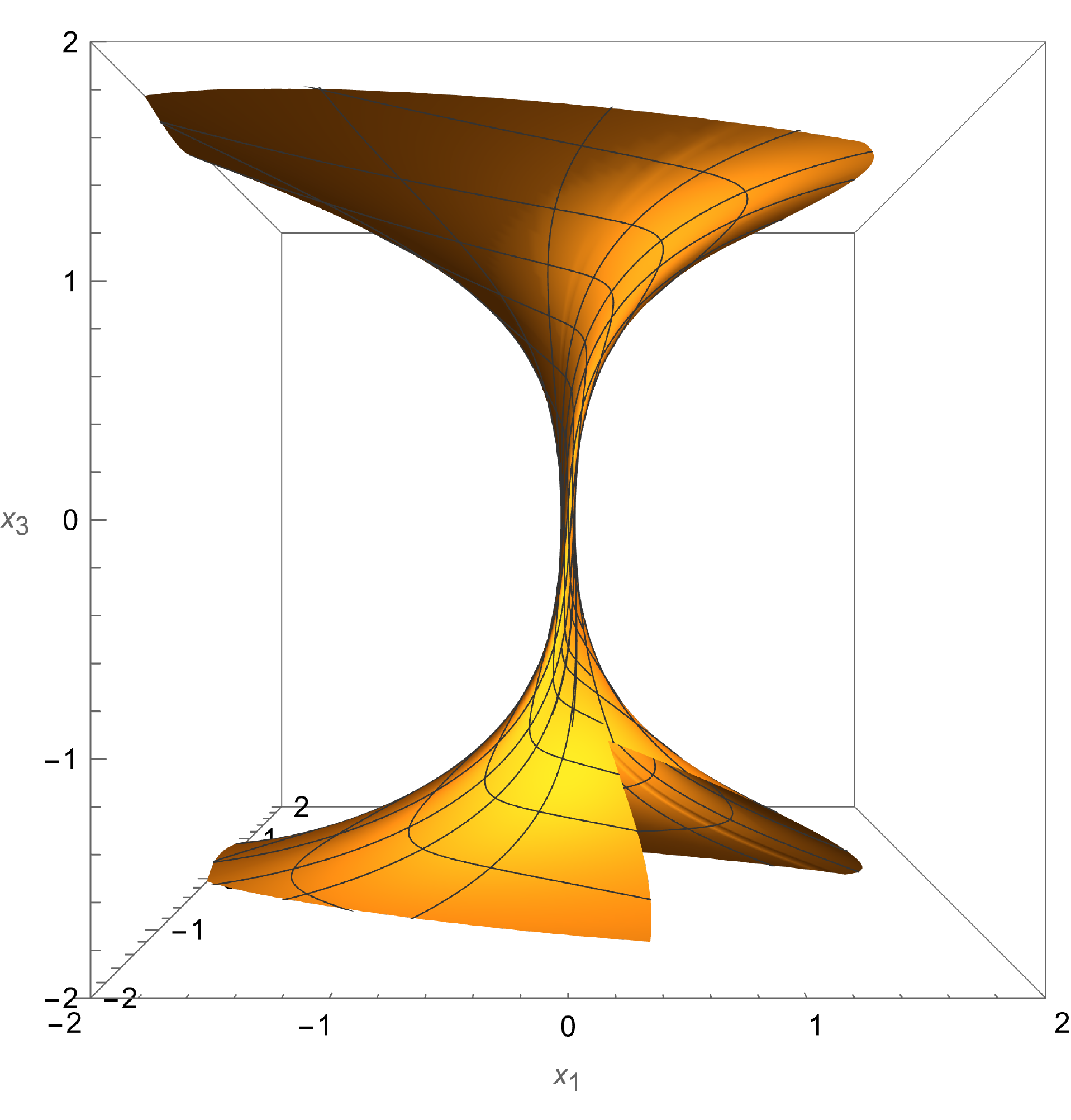}\\[-0.5cm]
    \caption{Figure 1 }
  \end{subfigure}
    \begin{subfigure}{0.49\textwidth}
    \includegraphics[width=0.95\textwidth]{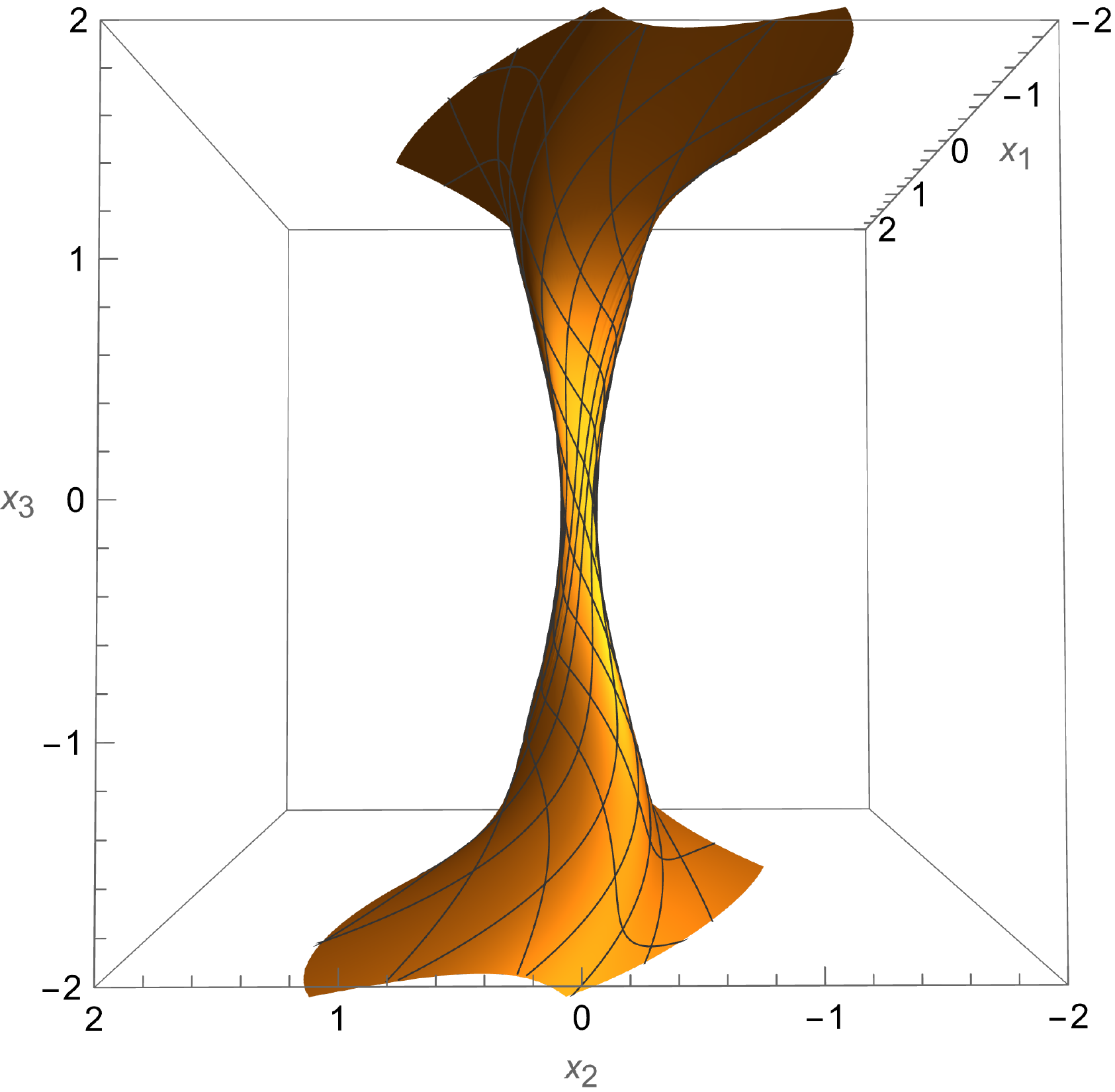}\\[-0.5cm]
    \caption{Figure 2 }
  \end{subfigure}\\[5mm]
    \caption{Catenoid with $\lambda_1=5$, $\lambda_2=1$, $c=2$, created by Mathematica.}
 \label{fig:simulation3}
\end{figure}

%\begin{figure}[H]
  %\centering
%\begin{subfigure}{0.49\textwidth}
    %\includegraphics[width=1\textwidth]{./figure/sandiantuzheng.pdf}\\[-0.5cm]
    %\caption{Case 1 }
  %\end{subfigure}
    %\begin{subfigure}{0.49\textwidth}
    %\includegraphics[width=0.95\textwidth]{./figure/sandiantuding.pdf}\\[-0.5cm]
   % \caption{Case 2 }
  %\end{subfigure}\\[5mm]
    %\caption{newnewnew.}
 %\label{fig:simulation3}
%\end{figure}

%\begin{figure}[H]
  %\centering
      %includegraphics[width=0.9\textwidth]{./figure/1.pdf}
          %\caption{Visualization of TCGA-AG-3574 tumor data}
 %\label{fig:GC}
  %\end{figure}

 \section{The limit of catenoids and a half-space theorem in $\widetilde{E(2)}$}

The goal of this section is to study the limit of catenoids $\mathfrak{C}_c$ when the parameter $c$ tends to $+\infty$. Since the functions mentioned in Section 5 depend on the parameter $c$, we will denote them by $\varphi_c,f_c,U_c, D_c,G_c,A_c$, etc. As an application, we get a half-space theorem for minimal surfaces in $\widetilde{E(2)}$.

\begin{proposition}\label{prop 6.1}
Suppose that $(\Tilde{u},\Tilde{v})\in\mathbb{R}^2$ and $c>0$. Let us make a change of variables 
\begin{align*}
    u_c = \frac{\Tilde{u}}{c},\quad v_c = \frac{2\ln{c}+\Tilde{v}}{c}.
\end{align*}
Then when $c\rightarrow+\infty$, the catenoid $\mathfrak{C}_c$ has the limit 
\begin{align*}
\lim\limits_{c\to +\infty}X_c\left(u_c + iv_c\right)
     =\begin{pmatrix}
        -\frac{1}{2\lambda_1}\cos{\Tilde{u}}e^{\Tilde{v}}\\
        -\frac{1}{2\lambda_2}\sin{\Tilde{u}}e^{\Tilde{v}}\\
        0
        \end{pmatrix}.
\end{align*}
\end{proposition}
\begin{proof}
We have seen in Lemma \ref{lem 5.3} that $\lim\limits_{c\to +\infty}\widetilde{\theta_{c}} = \frac{\pi}{2}$, hence $D_c\sim\frac{1}{c}$. By utilizing the formula (\ref{5.2}), we can estimate 
\begin{align*}
    \left|(f_c)'(u)\right| = \left|D_{c}B_{c}\right|\leq\frac{\lambda_1^2}{c},
\end{align*}
whence $|f_{c}|\leq\frac{\lambda_1^2}{c}|u|$ and $f_{c}(u_{c})=O(\frac{1}{c^2})$. Similarly, from (\ref{5.8}) we get
\begin{align*}
    \left|(G_{c})'(u)\right| =\left|\frac{D_{c}^2B_{c}-2\cos{\widetilde{\theta_{c}}}}{c+{\varphi_{c}'(u)}}\right|\leq\frac{\lambda_1^2}{c^3}+\frac{2\cos{\widetilde{\theta_{c}}}}{c},
\end{align*}
which implies $G_{c}(u_{c})=O(\frac{1}{c^2})$. By the definition of $A_c$, we have 
\begin{align*}
    A_c(u_c+iv_c) = 2\ln{c}+\Tilde{v}+O(\frac{1}{c^2}),
\end{align*}
thus $\cosh{A_c(u_c+iv_c)} \sim\frac{1}{2}c^2e^{\Tilde{v}}$ and $\sinh{A_c(u_c+iv_c)} \sim\frac{1}{2}c^2e^{\Tilde{v}}$. As for the function $\varphi_{c}$, when $c>\lambda_1^2\geq1$, we have 
\begin{align*}
    \sqrt{c^2-\lambda_1^4}\leq{\varphi_c}'(u)\leq\sqrt{c^2+2\lambda_1^2}.
\end{align*}
Hence 
\begin{align*}
     \sqrt{c^2-\lambda_1^4}u_{c}\leq\varphi_c(u_{c})\leq\sqrt{c^2+2\lambda_1^2}u_{c}.
\end{align*}
This means $\varphi_c(u_c)\rightarrow\Tilde{u}$. In addition, the formula (\ref{5.1}) tells us 
\begin{align*}
    c-{\varphi_c}'(u) = \frac{D_{c}^2B_{c}^2 - 2\cos{\widetilde{\theta_{c}}}B_c}{c+{\varphi_c}'(u)}.
\end{align*}
Since 
\begin{align*}
    0\leq c\left|\frac{D_{c}^2B_{c}^2 - 2\cos{\widetilde{\theta_{c}}}B_c}{c+{\varphi_c}'(u)}\right|\leq \left|D_{c}^2B_{c}^2 - 2\cos{\widetilde{\theta_{c}}}B_c\right|\leq\lambda_1^2\left|D_{c}^2B_{c} - 2\cos{\widetilde{\theta_{c}}}\right|,
\end{align*}
we obtain $c-{\varphi_c}'(u_c)=o(\frac{1}{c})$. Thanks to all of these preparations, we may determine the limits of coordinate functions given in Theorem \ref{thm 5.1} to be 
\begin{equation*}
\left\{
\begin{aligned}
  & ~\lim\limits_{c\to +\infty}(x_1)_c(u_c+iv_c) = -\frac{1}{2\lambda_1}\cos{\Tilde{u}}e^{\Tilde{v}},\\
  & ~\lim\limits_{c\to +\infty}(x_2)_c(u_c+iv_c) = -\frac{1}{2\lambda_2}\sin{\Tilde{u}}e^{\Tilde{v}},\\
  & ~\lim\limits_{c\to +\infty}(x_3)_c(u_c+iv_c) = 0.
\end{aligned}
\right.
\end{equation*}
The proof is completed.
\end{proof}

\begin{remark}\label{rmk 6.1}
The geometric signification of this result is that the part of the catenoid $\mathfrak{C}_c$ with $A>0$ converges to the punctured plane $\{x_3=0\}\setminus\{(0,0,0)\}$ as the parameter $c$ tends to $+\infty$. Using a similar change of variables, we may prove that when $c\rightarrow+\infty$, the part of the catenoid $\mathfrak{C}_c$ with $A<0$ converges to the punctured plane $\{x_3=0\}\setminus\{(0,0,0)\}$ as well. With the same argument, we can see that the half of the catenoid $\mathfrak{C}_c$ with $x_3\geq0$ also converges to $\{x_3=0\}\setminus\{(0,0,0)\}$ as the parameter $c$ tends to $+\infty$.
\end{remark}

\begin{proposition}\label{prop 6.2}
When $c\rightarrow+\infty$, the intersection curve of the catenoid $\mathfrak{C}_c$ with the plane $\{x_3 = 0\}$ converges uniformly to the origin.
\end{proposition}
\begin{proof}
The condition $(x_3)_c=0$ implies $v = \frac{G_c(u)}{D_c}$, thus the function $A$ can be written as 
\begin{align*}
    A_c(u) = f_c(u)+\frac{c}{D_c}G_c(u).
\end{align*}
As we have seen in the proof of Proposition \ref{prop 6.1}, we have $\left|(f_c)'(u)\right| \leq\frac{\lambda_1^2}{c}$ and
\begin{align*}
    \left|c-{\varphi_c}'(u)\right| = \left|\frac{D_{c}^2B_{c}^2 - 2\cos{\widetilde{\theta_{c}}}B_c}{c+{\varphi_c}'(u)}\right|\leq\frac{\lambda_1^2}{c}\left(\frac{\lambda_1^2}{c^2}+2\right).
\end{align*}
Therefore, with the formulas in the proof of Proposition \ref{prop 5.5}, one can get the following estimations 
\begin{align*}
    \left|(x_1)_c(u)\right| &\leq\frac{1}{\left(c^2+\lambda_1^2\lambda_2^2D_{c}^2\right)B_c}\left[\left(\frac{c}{\lambda_1}\left|f_{c}'(u)\right|+\lambda_1D_c\left|c-{\varphi_c}'(u)\right|\right)\cosh{A_c(u)}\right.\\
    &~~~~~\quad \left.+ \left(\frac{c}{\lambda_1}\left|c-{\varphi_c}'(u)\right|+\lambda_1D_c\left|f_{c}'(u)\right|\right)\sinh{\left|A_c(u)\right|}\right]\\
    &\leq \frac{1}{\lambda_2^2c^2}\left[\left(\lambda_1+\frac{\lambda_1^3}{c^2}\left(\frac{\lambda_1^2}{c^2}+2\right)\right)\cosh{A_c(u)}\right.\\
    &~~~~~\quad\left. + \left(\lambda_1\left(\frac{\lambda_1^2}{c^2}+2\right)+\frac{\lambda_1^3}{c^2}\right)\sinh{\left|A_c(u)\right|}\right]
\end{align*}
and 
\begin{align*}
    \left|(x_2)_c(u)\right| &\leq\frac{1}{\left(c^2+\lambda_1^2\lambda_2^2D_{c}^2\right)B_c}\left[\left(\frac{c}{\lambda_2}\left|f_{c}'(u)\right|+\lambda_2D_c\left|c-{\varphi_c}'(u)\right|\right)\cosh{A_c(u)}\right.\\
    &~~~~~\quad \left.+ \left(\frac{c}{\lambda_2}\left|c-{\varphi_c}'(u)\right|+\lambda_2D_c\left|f_{c}'(u)\right|\right)\sinh{\left|A_c(u)\right|}\right]\\
    &\leq \frac{1}{\lambda_2^2c^2}\left[\left(\frac{\lambda_1^2}{\lambda_2}+\frac{\lambda_1^2\lambda_2}{c^2}\left(\frac{\lambda_1^2}{c^2}+2\right)\right)\cosh{A_c(u)}\right.\\
    &~~~~~\quad\left. + \left(\frac{\lambda_1^2}{\lambda_2}\left(\frac{\lambda_1^2}{c^2}+2\right)+\frac{\lambda_1^2\lambda_2}{c^2}\right)\sinh{\left|A_c(u)\right|}\right].
\end{align*}
It is known from Theorem \ref{thm 5.2} that the function $A_c(u)$ is $2U_c$-periodic with 
\begin{align*}
    U_c = \mathlarger{\int}_{-1}^{1}\frac{{\ud}x}{\sqrt{(1-x^2)P_c(x)}}.
\end{align*}
When $c>\lambda_1^2\geq1$, we always have $P_c(x)\geq c^2-\lambda_1^4$, whence $U_c\leq\frac{\pi}{\sqrt{c^2-\lambda_1^4}}$. As a consequence, we get
\begin{align*}
    \left|A_c(u)\right|\leq\frac{\lambda_1^2\pi}{c\sqrt{c^2-\lambda_1^4}} + \frac{3\pi c}{\sin{\widetilde{\theta_{c}}}\sqrt{c^2-\lambda_1^4}}.
\end{align*}
Since $\lim\limits_{c\to +\infty}\sin{\widetilde{\theta_{c}}}=1$, we can see that the function $A_c$ is uniformly bounded when $c\rightarrow+\infty$. Therefore, the coordinate functions $(x_1)_c$ and $(x_2)_c$ are uniformly bounded as well when $c\rightarrow+\infty$. The conclusion can be drawn by taking the limits. 
\end{proof}

Thanks to all these preparations, we are able to prove the following half-space theorem for minimal surfaces in $\widetilde{E(2)}$.

\begin{theorem}\label{thm 6.1}
Suppose that $\Sigma$ is a properly immersed minimal surface in $\widetilde{E(2)}$ which is totally contained on the one side of a plane $Y$ parallel to $\{x_3 = 0\}$, then this surface $\Sigma$ must be a plane parallel to $Y$.  
\end{theorem}
\begin{proof}
By applying an isometry of $\widetilde{E(2)}$ if necessary, i.e., a left multiplication of an element $(0,0,a)$, we may identify this plane $Y$ with $\{x_3 = 0\}$. Moreover, we can also suppose that $\Sigma\subset\{(x_1,x_2,x_3)\mid x_3\leq0\}$ and that $\Sigma$ is not contained in any half-space $\{(x_1,x_2,x_3)\mid x_3\leq-\epsilon\}$ for every $\epsilon>0$.

Let us assume that $\Sigma$ is not a plane parallel to $Y$, then the maximum principle tells us that $\Sigma\cap Y=\varnothing$. For a given $\epsilon>0$, we define a mapping 
\begin{align*}
    T_{\epsilon}:\widetilde{E(2)}&\longrightarrow \widetilde{E(2)},\\
    (x_1,x_2,x_3)&\mapsto(x_1\cos{\epsilon}-x_2\sin{\epsilon}, x_1\sin{\epsilon}+x_2\cos{\epsilon},x_3+\epsilon).
\end{align*}
This mapping is actually the left multiplication of the element $(0,0,\epsilon)$, thus an isometry of $\widetilde{E(2)}$. By the hypothesis we made above, we must have $T_{\epsilon}(\Sigma)\cap Y\neq\varnothing$ for $\epsilon$ small enough.

Let $c\geq1$, we denote the upper half catenoid by ${\mathfrak{C}'_{c}} = \mathfrak{C}_{c} \cap\{(x_1,x_2,x_3)\mid x_3\geq0\}$. Since the intersection curve of the catenoid $\mathfrak{C}_c$ with the plane $Y$ converges uniformly to the origin when $c\rightarrow+\infty$, we can find a compact $\Delta\subset Y$ which contains $0$ and $\mathfrak{C}_c\cap Y$ for all $c\geq1$. Then there exists a $\epsilon_{0}>0$ such that 
\begin{align*}
    T_{\epsilon_{0}}(\Sigma)\cap Y\neq\varnothing,\quad T_{\epsilon_{0}}(\Sigma)\cap \mathfrak{C}'_{1}=\varnothing,\quad T_{\epsilon_{0}}(\Sigma)\cap \Delta=\varnothing.
\end{align*}
We will prove this announcement by contradiction. If it is not true, then there is a sequence $(\epsilon_n)$ of positive real numbers which converges to $0$, together with a sequence of points $(p_n)$ such that $p_n\in T_{\epsilon_{n}}(\Sigma)\cap\left(\mathfrak{C}'_{1}\cup\Delta\right)$. When $n$ is sufficiently large, $p_n$ falls in the union of $\Delta$ and the part of $\mathfrak{C}'_{1}$ satisfying $0\leq x_3\leq 1$, which is a compact set. Therefore, we can suppose that the sequence $(p_n)$ has a limit point $p$. It is easy to see that $p\in\Delta\cap \Sigma$, which is not possible.

We fix the $\epsilon_{0}>0$ as above. From Proposition \ref{prop 6.1}, we have known that ${\mathfrak{C}'_{c}}$ converges smoothly to $Y\setminus\{0\}$ when $c\rightarrow+\infty$. As a result of this fact and the maximum principle, we may say that $T_{\epsilon_{0}}(\Sigma)\cap \mathfrak{C}'_{c}\neq\varnothing$ when $c$ is sufficiently large. On the other side, since $ T_{\epsilon_{0}}(\Sigma)\cap \mathfrak{C}'_{1}=\varnothing$, we must have $T_{\epsilon_{0}}(\Sigma)\cap \mathfrak{C}'_{c}=\varnothing$ for $c$ close enough to $1$. It is very natural to consider the non-empty set $\Xi = \{c\geq1\mid T_{\epsilon_{0}}(\Sigma)\cap \mathfrak{C}'_{c}\neq\varnothing\}$ and $\xi = \inf\limits_{c\in\Xi}{c}$. Clearly we get $\xi>1$. We are going to show that $\xi\in\Xi$.

It is trivial when $\xi$ is an isolated point. So we may focus on the situation when $\xi$ is not isolated. In this case, there is a decreasing sequence $(c_n)$ which converges to $\xi$ and a sequence of points $(q_n)$ with $q_n\in T_{\epsilon_{0}}(\Sigma)\cap \mathfrak{C}'_{c_n}$. Using Proposition \ref{prop 5.7}, we can write $q_n=X_{c_n}(u_n+iv_n)$ with $u_n\in[-U_{c_n},U_{c_n}]$ and $v_n\in\mathbb{R}$. Then the hypothesis $0\leq (x_3)_{c_n}(u_n+iv_n)\leq\epsilon_0$ is equivalent to 
\begin{align*}
    0\leq -D_{c_n}v_n + G_{c_n}(u_n)\leq\frac{\epsilon_0}{\lambda_1\lambda_2}.
\end{align*}
The condition $\xi\leq c_n\leq c_1$ implies $u_n$ is bounded, thus $G_{c_n}(u_n)$ is also bounded. It follows that $v_n$ is bounded as well. Since the sequence $(u_n+iv_n)$ is contained in a compact set of $\mathbb{C}$, it possesses a sub-sequence which has a limit point $u+iv\in\mathbb{C}$. Without loss of generality, we may consider this sub-sequence instead of the original one. Hence the sequence $q_n$ converges to a point $q\in T_{\epsilon_{0}}(\Sigma)\cap \mathfrak{C}'_{\xi}$. This means $\xi\in\Xi$.

By the definition of $\Xi$, there is a point $r\in T_{\epsilon_{0}}(\Sigma)\cap \mathfrak{C}'_{\xi}$. This point $r$ must be an interior point of $\mathfrak{C}'_{\xi}$ due to $\partial\mathfrak{C}'_{\xi}\subset\Delta$ and $T_{\epsilon_{0}}(\Sigma)\cap \Delta=\varnothing$. The fact that $ T_{\epsilon_{0}}(\Sigma)\cap \mathfrak{C}'_{c}=\varnothing$ for all $c<\xi$ implies that $\mathfrak{C}'_{\xi}$ remains on the one side of $T_{\epsilon_{0}}(\Sigma)$ in a small neighbourhood of $r$. It follows from the maximum principle that $T_{\epsilon_{0}}(\Sigma)=\mathfrak{C}_{\xi}$. This is impossible because we know that $\mathfrak{C}_{\xi}$ is not contained in the half-space $\{(x_1,x_2,x_3)\mid x_3\leq0\}$. This completes the proof.
\end{proof}

\begin{remark}\label{rmk 6.2}
The original idea of this proof is due to Hoffman and Meeks \cite{hoffman1990strong}. They proved the half-space theorem for minimal surfaces in $\mathbb{R}^3$ which can be regarded as the maximum principle at infinity. Daniel and Hauswirth obtained in their article \cite{daniel2009half} a vertical half-space theorem for minimal surfaces in the Heisenberg group $\textrm{Nil}_3$ by using a similar method. 
\end{remark}

\begin{remark}\label{rmk 6.3}
Mazet has proved a general half-space theorem for constant mean curvature surfaces in three-dimensional Lie groups (see \cite{mazet2013general}, Proposition 10). From this viewpoint, Theorem \ref{thm 6.1} is new in its proof rather than the announcement itself.    
\end{remark}

%%%%%%%%%%%%%%%%%%%%%%%%%%%%%%%%%%%%%%%%%%%%%%%%%%%%%%%%%%%%%%%%%%%%%%%%%%%%%%%%%%%%%%%%%%%%%%%%%%%%%%%%%%%%%%%%%%%%%%%%%%%
%\section*{Acknowledgements}
% The authors thank the editor, the associate editor, and two referees for their valuable comments, which have led to the improvement of the paper.
%\par

%%%%%%%%%%%%%%%%%%%%%%%%%%%%%%%%%%%%%%%%%%%%%%%%%%%%%%%%%%%%%%%%%%%%%%%%%%%%%%%%%%%%%%%%%%%%%%%%%%%%%%%%%%%%%%%%%%%%%%%%%%%
%\section*{Acknowledgements}
% The authors thank the editor, the associate editor, and two referees for their valuable comments, which have led to the improvement of the paper.
%\par

\bibliographystyle{plain}
\bibliography{Catenoid}

%\vspace{.5 cm} 
%\noindent Université de Lorraine, CNRS, IECL, F-54000 Nancy, France\\
%   \emph{Email address}:~yiming.zang@univ-lorraine.fr

\end{document}